\documentclass[aos,MSNbibl,citesort,noautosecdot,dvips]{arximspdf}
\usepackage{mathbh,subeqn}
\usepackage{graphicx}

\doi{10.1214/11-AOS878}
\volume{39}
\issue{3}
\pubyear{2011}
\firstpage{1335}
\lastpage{1371}

\makeatletter
\renewcommand{\epsilon}{\varepsilon}

\newproclaim{algorithm}{Algorithm}
\newtheorem{theorem}{Theorem}
\newtheorem{lemma}{Lemma}
\newtheorem{corollary}{Corollary}

\newcommand{\eqref}[1]{(\ref{#1})}

\newcommand{\argmax}{\operatorname{\arg\max}\limits}
\newcommand{\argmin}{\operatorname{\arg\min}\limits}
\newcommand{\minimize}{\operatorname{minimize}\limits}
\def\E{\mathrm{E}}
\def\sign{\operatorname{sign}}
\def\tr{\operatorname{tr}}
\def\half{\frac{1}{2}}
\def\halft{\frac{1}{2}}
\def\halftt{\tfrac{1}{2}}
\def\hbeta{\hat{\beta}}
\def\hu{\hat{u}}
\def\tbeta{\tilde{\beta}}
\def\tu{\tilde{u}}
\def\ty{\tilde{y}}
\def\tD{\widetilde{D}}
\def\lone{1}
\def\ltwo{2}
\def\linf{\infty}
\def\cB{\mathcal{B}}
\def\cG{\mathcal{G}}
\def\cM{\mathcal{M}}
\def\cN{\mathcal{N}}
\def\T{^T}
\def\df{\operatorname{df}}
\def\R{\mathbb{R}}
\setattribute{keyword}{AMS}{AMS 2010 subject classification.}
\makeatother

\begin{document}
\begin{frontmatter}

\title{The solution path of the generalized lasso}
\runtitle{The generalized Lasso}

\begin{aug}
\author[A]{\fnms{Ryan J.} \snm{Tibshirani}\corref{}\ead[label=e1]{ryantibs@stanford.edu}}
\and
\author[A]{\fnms{Jonathan} \snm{Taylor}\ead[label=e2]{jtaylo@stanford.edu}}
\runauthor{R. J. Tibshirani and J. Taylor}
\affiliation{Stanford University}
\address[A]{Department of Statistics\\
Stanford University\\
Stanford, California 94305\\
USA\\
\printead{e1}\\
\hphantom{\textsc{E-mail:}\ }\printead*{e2}} 
\end{aug}

\received{\smonth{8} \syear{2010}}
\revised{\smonth{1} \syear{2011}}

%
\begin{abstract}
We present a path algorithm for the generalized lasso problem. This
problem penalizes the $\ell_1$ norm of a matrix $D$ times the
coefficient vector, and has a wide range of applications,
dictated by the choice of $D$. Our algorithm is based on
solving the dual of the generalized lasso, which
greatly facilitates computation of the path.
For $D=I$ (the usual lasso), we draw a connection between our
approach and the well-known LARS algorithm.
For an arbitrary~$D$, we derive an unbiased estimate of the degrees
of freedom of the generalized lasso fit. This estimate turns out to
be quite intuitive in many applications.
\end{abstract}

%

\end{frontmatter}

\section{\texorpdfstring{Introduction.}{Introduction}}\label{intro}
\label{sec:intro}

Regularization with the $\ell_1$ norm seems to be ubiquitous
throughout many fields of mathematics and engineering. In statistics,
the best-known example is the \textit{lasso}, the application of an
$\ell_1$ penalty to linear regression~\cite{lasso,bp}. Let $y \in
\R^n$ be a response vector and $X \in\R^{n \times p}$ be a matrix of
predictors. If the response and the predictors have been
centered, we can omit an intercept term from the model, and then
the lasso problem is commonly written as
%
\begin{equation}
\label{eq:lasso}
\minimize_{\beta\in\R^p}   \halft\|y-X\beta\|_\ltwo^2 +
\lambda\|\beta\|_\lone,
\end{equation}
where $\lambda\geq0$ is a tuning parameter. There are many fast
algorithms for solving the lasso~\eqref{eq:lasso} at a single
value of the parameter $\lambda$, or over a discrete set of parameter
values. The \textit{least angle regression} (LARS) algorithm, on the
other hand, is unique in that it solves~\eqref{eq:lasso} for all
$\lambda\in[0,\infty]$~\cite{lars} (see also the earlier \textit{homotopy} method of~\cite{homotopy}, and the even earlier work of
\cite{parqp}). This is possible because the lasso solution is
piecewise linear with respect to $\lambda$.

The LARS path algorithm may provide a computational advantage when
the solution is desired at many values of the tuning parameter. For
large problems, this is less likely to be the case because
the number of knots (changes in slope) in the solution path tends to
be very large, and this renders the path intractable.
Computational efficiency aside, the LARS method fully characterizes
the tradeoff between goodness-of-fit and sparsity in the lasso
solution (this is controlled by $\lambda$), and hence
yields interesting statistical insights into the problem. Most
notably, the LARS paper established a result on the degrees of freedom
of the lasso fit, which was further developed by
\cite{lassodf}.

The first of its kind, LARS inspired
the development of path algorithms for various other optimization
problems that appear in statistics~\cite{svmpath,rosset,holger,dasso},
and our case is no exception. In this paper, we derive a path
algorithm for problems that use the $\ell_1$ norm to enforce certain
structural constraints---instead of pure sparsity---on the
coefficients in a linear regression. These problems are nicely
encapsulated by the formulation:
%
\begin{equation}
\label{eq:dlasso}
\minimize_{\beta\in\R^p}   \halft\|y-X\beta\|_\ltwo^2 +
\lambda\|D\beta\|_\lone,
\end{equation}
where $D \in\R^{m \times p}$ is a specified penalty
matrix.
We refer to problem~\eqref{eq:dlasso} as the \textit{generalized lasso}.
Depending on the application, we choose $D$ so that
sparsity of $D\beta$ corresponds to some other desired behavior for
$\beta$, typically one that is structural or geometric in nature.
In fact, various choices of $D$ in~\eqref{eq:dlasso} give problems
that are already well-known in the literature: the fused
lasso, trend filtering, wavelet smoothing, and a method for outlier
detection. We derive a simple path algorithm for the minimization
\eqref{eq:dlasso}
that applies to a general matrix $D$, hence this entire class of
problems. Like the lasso, the generalized lasso solution is piecewise
linear as a function of $\lambda$. We also prove a result on the
degrees of freedom of the fit for a general $D$. It is worth noting
that problem~\eqref{eq:dlasso} has been considered by other
authors, for example,~\cite{tisp}. This last work establishes some
asymptotic properties of the solution, and
proposes a computational technique that relates to simulated
annealing.

The paper is organized as follows.
We begin in Section~\ref{sec:apps} by motivating the use of a penalty
matrix $D$, offering several examples of problems that fit into this
framework. Section~\ref{sec:lasso} explains that some instances of the
generalized lasso can be transformed into a regular lasso problem, but
many cannot, emphasizing the need for a new path approach.
In Section~\ref{sec:dual}, we derive the Lagrange dual
of~\eqref{eq:dlasso}, which serves as the jumping point for our
algorithm and all
of the work that follows. For the sake of clarity, we build up the
algorithm over the next 3 sections. Sections~\ref{sec:1d} and
\ref{sec:d} consider the case $X=I$. In Section~\ref{sec:1d}, we assume
that $D$ is the 1-dimensional fused lasso matrix, in which case our
path algorithm takes an especially simple (and intuitive) form. In
Section~\ref{sec:d}, we give the path algorithm for a general penalty
matrix $D$, which requires adding only one step in the iterative
loop. Section~\ref{sec:x} extends the algorithm to the case of a
general design matrix $X$. Provided that $X$ has full column rank, we
show that our path algorithm still applies, by rewriting the dual
problem in a more familiar form. We also outline a path approach for
the case when $X$ has rank less than its number of
columns. Practical considerations for the path's computation
are given in Section~\ref{sec:comp}.\looseness=1

In Section~\ref{sec:lars}, we focus on the lasso case, $D=I$,
and compare our method to LARS. Above, we described LARS as
an algorithm for computing the solution path of~\eqref{eq:lasso}.
This actually refers to LARS in its ``lasso'' state, and
although this is probably the best-known version of LARS, it is not
the only one. In its original (unmodified) state, LARS does not
necessarily optimize the lasso criterion, but instead performs a more
``democratic'' form of forward variable selection. It turns out
that with an easy modification, our algorithm gives this selection
procedure exactly. In Section~\ref{sec:df}, we derive an
unbiased estimate of the degrees of freedom of the fit for a general
matrix $D$.
The proof is quite straightforward because it utilizes the dual
fit, which is simply the projection onto a convex set. As we vary $D$,
this result yields interpretable estimates of the degrees of freedom
of the fused lasso, trend filtering, and more. Finally, Section
\ref{sec:discuss} contains some discussion.

To save space (and improve readability), many of the technical details
in the paper are deferred to a supplementary document~\cite{supp}.

\section{\texorpdfstring{Applications.}{Applications}}
\label{sec:apps}

There are a wide variety of interesting applications of problem
\eqref{eq:dlasso}. What we present below is not meant to be an
exhaustive list, but rather a set of illustrative examples that
motivated our work on this problem in the first place. This section is
split into two main parts: the case when $X=I$ (often called the
``signal approximation'' case), and the case when $X$ is a general
design matrix.

\subsection{\texorpdfstring{The signal approximation case, $X=I$.}{The signal approximation case, $X=I$}}

When $X=I$, the solution of the lasso problem~\eqref{eq:lasso} is
given by soft-thresholding the coordinates of $y$. Therefore, one
might think that an equally simple formula exists for the generalized
lasso solution when the design matrix is the identity---but this is
not true. Taking $X=I$ in the generalized lasso~\eqref{eq:dlasso}
gives an interesting and highly nontrival class of problems. In this
setup, we observe data $y \in\R^n$ which is a noisy realization of an
underlying signal, and the rows of $D \in\R^{m \times n}$ reflect
some believed structure or geometry in the signal. The solution of
problem~\eqref{eq:dlasso} fits adaptively to the data while exhibiting
some of these structural properties. We begin by looking at piecewise
constant signals, and then address more complex features.

\subsubsection{\texorpdfstring{The fused lasso.}{The fused lasso}}

Suppose that $y$ follows a 1-dimensional structure, that is,
the coordinates of $y$ correspond to successive positions on a
straight line. If $D$ is the $(n-1)\times n$ matrix
%
\begin{equation}
\label{eq:d1d}
D_\mathrm{1d} = \left[
\matrix{
-1 & 1 & 0 & \cdots & 0 & 0 \cr
0 & -1 & 1 & \cdots & 0 & 0 \cr
& & & \cdots & & \cr
0 & 0 & 0 & \cdots & -1 & 1
}
\right],
\end{equation}
then problem~\eqref{eq:dlasso} penalizes the absolute differences in
adjacent coordinates of~$\beta$, and is known as the \textit{1d fused
lasso}~\cite{fuse}. This gives a piecewise constant fit, and is used
in settings where coordinates in the true model are closely related to
their neighbors. A common application area is comparative genomic
hybridization (CGH) data: here $y$ measures the number of copies of
each gene ordered linearly along the genome (actually $y$ is the log
ratio of the number of copies relative to a normal sample),
and we believe for biological reasons that nearby genes will exhibit a
similar copy number.
Identifying abnormalities in copy number has become a
valuable means of understanding the development of many human
cancers.
See Figure~\ref{fig:cgh} for an example of the 1d fused lasso
applied to some CGH data on glioblastoma multiformes, a particular
type of malignant brain tumor, taken from~\cite{gbm}.

%
\begin{figure}

\includegraphics{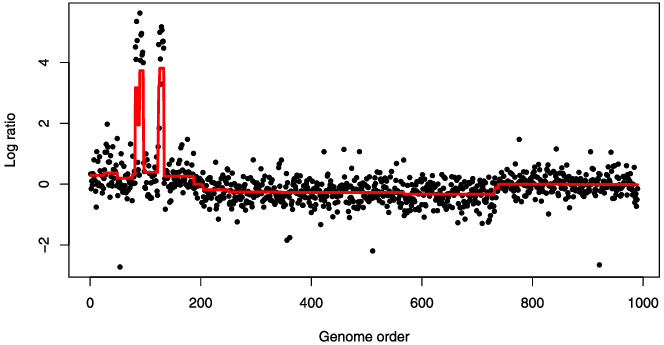}

\caption{The 1d fused lasso applied to some glioblastoma
multiforme data. The red line represents the inferred copy number
from the 1d fused lasso solution (for $\lambda=3$).}
\label{fig:cgh}
\end{figure}

A natural extension of this idea penalizes the differences between
neighboring pixels in an image.
Suppose that $y$ represents a noisy image that has been
unraveled into a vector, and each row of $D$ again has a $1$ and $-1$,
but this time arranged to give both the horizontal and vertical
differences between pixels. Then problem~\eqref{eq:dlasso}
is called the \textit{2d fused lasso}~\cite{fuse}, and is used to denoise
images that we believe should obey a piecewise constant
structure. This technique is actually a special type of
\textit{total variation denoising},
a well-studied problem that carries a vast literature spanning the
fields of statistics, computer science, electrical engineering, and
others (see~\cite{tv}, e.g.). Figure~\ref{fig:s} shows the 2d
fused lasso applied to a toy example.

%
\begin{figure}

\includegraphics{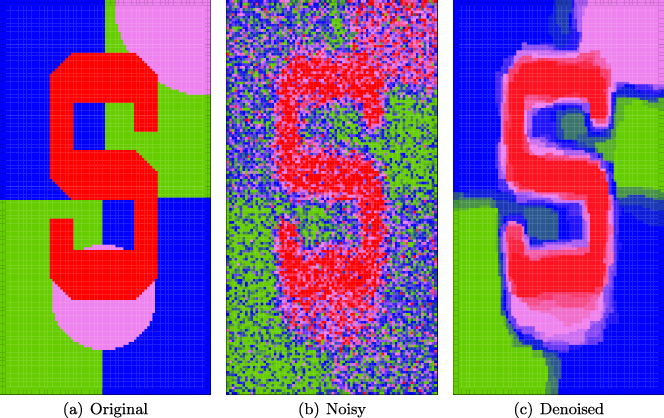}

\caption{An example of the 2d fused lasso for image denoising. We started
with a toy signal, shown in \textup{(a)}. The colors green, blue,
purple, red
in the image correspond to the numeric levels $1,2,3,4$,
respectively. We then added noise, shown in \textup{(b)},
interpolating between colors to display the intermediate
values. This is used as the data $y$ in the 2d fused lasso
problem. The solution (for $\lambda=1$) is shown in \textup{(c)}, and
it is a
fairly accurate reconstruction. The method is effective here because
the original image is piecewise constant.}
\label{fig:s}
\end{figure}

We can further extend this idea by defining adjacency according to an
arbitrary graph structure, with $n$ nodes and $m$ edges.
Now the coordinates of $y \in\R^n$ correspond to nodes in the graph,
and we penalize the difference between each pair of nodes joined by an
edge. Hence $D$ is $m \times n$, with each row having a $-1$
and $1$ in the appropriate spots, corresponding to an edge in the
graph. In this case, we simply call problem~\eqref{eq:dlasso} the
\textit{fused lasso}. Note that both the 1d and 2d fused lasso problems
are special cases of this, with the underlying graph
a chain and a 2d grid, respectively. But the fused lasso is a
very general problem, as it can be applied to any graph structure that
exhibits a piecewise constant signal across adjacent nodes. See
Figure~\ref{fig:us} for application in which the underlying graph has
US states as nodes, with two states joined by an edge if they
share a border. This graph has 48 nodes (we only include the mainland
US states) and 105 edges.

%
\begin{figure}

\includegraphics{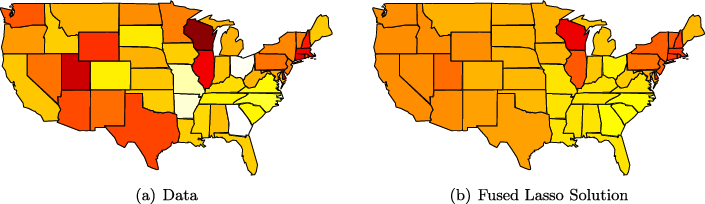}

\caption{An example of the fused lasso on an irregular
graph. The data $y$ are the log proportion of H1N1 flu cases
for each (mainland) US state in the year 2009, shown in~\textup{(a)}. This
was taken from \protect\cite{h1n1}. The color map uses
white to reflect the lowest measured log proportion, and dark red
to reflect the highest, with yellow, orange, and red in between.
We can think of the data as noisy measurements of the
true log probabilities of infection in each state, which likely
exhibits some geographic trend. Therefore, we solve the fused lasso
problem on a custom underlying graph, where we connect two states by
an edge if they share a border. Shown in \textup{(b)} is the solution
(for $\lambda=0.25$).
Here, groups of states are assigned the same color or
``fused'' on the west coast, in the mid west, in the south east, and
in the north east. The colors suggest that, among these regions, you
are most likely to get H1N1 flu if you live in the north east, then
the west coast, then the midwest, and then the south east. But there
certainly are states that do not get fused into these regions, like
Wisconsin and Illinois, where the infection rates are exceptionally
high.}
\label{fig:us}
\end{figure}

The observant reader may notice a discrepancy between the usual
fused lasso definition and ours, as the fused lasso penalty typically
includes an additional term $\|\beta\|_1$, the $\ell_1$ norm of the
coefficients themselves. We refer to this as the \textit{sparse fused
lasso}, and to represent this penalty we just append the
$n\times n$ identity matrix to the rows of $D$. Actually, this carries
over to all of the applications yet to be discussed---if we desire
pure sparsity in addition to the structural behavior that is being
encouraged by $D$, we append the identity matrix to the rows of
$D$.

\subsubsection{\texorpdfstring{Linear and polynomial trend filtering.}{Linear and polynomial trend filtering}}

Suppose again that $y$ follows a 1-dimensional
structure, but now $D$ is the $(n-2)\times n$ matrix
\[
\label{eq:dtf}
D_{\mathrm{tf},1} = \left[
\matrix{\displaystyle
-1 & 2 & -1 & \cdots & 0 & 0 & 0 \cr\displaystyle
0 & -1 & 2 & \cdots & 0 & 0 & 0 \cr\displaystyle
\cdots & & & & & & \cr\displaystyle
0 & 0 & 0 & \cdots & -1 & 2 & -1
}
\right].
\]
Then problem~\eqref{eq:dlasso} is equivalent to
\textit{linear trend filtering} (also called
\textit{$\ell_1$ trend filtering})~\cite{l1tf}.
Just as the 1d fused lasso penalizes the
discrete first derivative, this technique penalizes the discrete
second derivative, and so it gives a piecewise linear fit. This
has many applications, namely, any settings in which the underlying
trend is believed to be linear with (unknown) changepoints. Moreover,
by recursively defining
\[
\label{eq:Dtfk}
D_{\mathrm{tf},k} = D_\mathrm{1d} \cdot D_{\mathrm{tf},k-1}   \qquad
\mbox{for } k=2,3,\ldots
\]
[here $D_\mathrm{1d}$ is the $(n-k-1)\times(n-k)$ version of
\eqref{eq:d1d}], we can fit a piecewise polynomial of any order $k$,
further extending the realm of applications.
We call this \textit{polynomial trend filtering of order $k$}.
Figure~\ref{fig:tf} shows examples of linear, quadratic, and cubic
fits.

%
\begin{figure}

\includegraphics{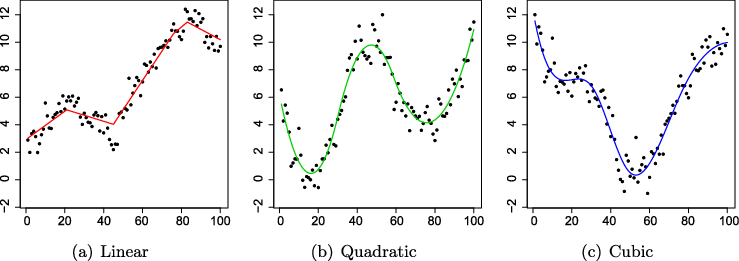}

\caption{Solutions of \protect\eqref{eq:dlasso} for three problems, with
$D$ equal to \textup{(a)} $D_{\mathrm{tf},1}$, \textup{(b)}
$D_{\mathrm{tf},2}$ and
\textup{(c)} $D_{\mathrm{tf},3}$.
These are piecewise linear, quadratic and
cubic, respectively. (For each problem we chose a
different value of the regularization parameter $\lambda$.)}
\label{fig:tf}
\end{figure}

The polynomial trend filtering fits (especially for $k=3$) are similar
to those that one could obtain using regression splines and smoothing
splines. However, the knots (changes in $k$th derivative) in the trend
filtering fits are selected adaptively based on the data, jointly with
the inter-knot polynomial estimation. This phenomenon of simultaneous
selection and estimation---analogous to that concerning the
nonzero coefficients in the lasso fit, and the jumps in the piecewise
constant fused lasso fit---does not occur in
regression and smoothing splines. Regression splines operate on a
fixed set of knots, and there is a substantial literature on
knot placement for this problem (see Chapter 9.3 of~\cite{gam}, e.g.). Smoothing splines place a knot at each data point, and
implement smoothness via a generalized ridge regression on the
coefficients in a natural spline basis. As a result (of this
$\ell_2$ shrinkage), they cannot represent both global smoothness and
local wiggliness in a signal. On the other hand, trend filtering has
the potential to represent both such
features, a~property called ``time and frequency localization''
in the signal processing field, though this idea has been largely
unexplored. The classic example of a procedure that allows
time and frequency localization is wavelet smoothing,
discussed next.

\subsubsection{\texorpdfstring{Wavelet smoothing.}{Wavelet smoothing}}

This is a quite a popular method
in signal processing and compression. The main idea is to model the
data as a sparse linear combination of wavelet functions. Perhaps
the most common formulation for wavelet smoothing is \textit{SURE
shrinkage}~\cite{sure}, which solves the lasso optimization problem
%
\begin{equation}
\label{eq:trans}
\minimize_{\theta\in\R^n}   \halft\|y-W\theta\|_\ltwo^2 +
\lambda\|\theta\|_\lone,
\end{equation}
where $W \in\R^{n\times n}$ has an orthogonal wavelet basis along its
columns. By orthogonality, we can change variables to $\beta=W\theta$
and then~\eqref{eq:trans} becomes a generalized lasso problem with
$D=W\T$.

In many applications it is desirable to use an overcomplete wavelet
set, so that $W \in\R^{n\times m}$ with $n < m$. Now problem
\eqref{eq:trans} and the generalized lasso~\eqref{eq:dlasso} with
$D=W\T$ (and $X=I$) are no longer equivalent, and in fact give quite
different answers. In signal processing, the former is called the
\textit{synthesis} approach, and the latter the \textit{analysis}
approach, to wavelet smoothing.
Though attention has traditionally been centered around
synthesis, a recent paper by Elad, Milanfar and Rubinstein~\cite{avs} suggests that
synthesis may be too sensitive, and shows that it
can be outperformed by its analysis counterpart.

\subsection{\texorpdfstring{A general design matrix $X$.}{A general design matrix $X$}}

For any of the fused lasso, trend filtering, or wavelet smoothing
penalties discussed above, the addition of a general matrix $X$ of
covariates significantly extends the domain of applications. For a
fused lasso example, suppose that each row of $X$ represents a
$k_1 \times k_2 \times k_3$ MRI image of a patient's brain,
unraveled into a vector (so that $p=k_1 \cdot k_3 \cdot k_3$).
Suppose that $y$ contains some continuous outcome on the patients, and
we model these as a linear function of the MRIs,
$\E(y_i|X_i)=\beta\T X_i$. Now $\beta$ also has
the structure of a $k_1 \times k_2 \times k_3$ image, and by
choosing the matrix $D$ to give the sparse 3d fused lasso penalty
(i.e., the fused lasso on a 3d grid with an additional $\ell_1$
penalty of the coefficients), the solution of~\eqref{eq:dlasso}
attempts to explain the outcome with a small number of contiguous
regions in the brain.

As another example, the inclusion of a design matrix $X$ in the trend
filtering setup provides an alternative way of fitting
\textit{varying-coefficient models}~\cite{vcm,lrm}.
We consider a data set from~\cite{vcm}, which examines
$n=88$ observations on the exhaust from an engine fueled by
ethanol. The response $y$ is the concentration of nitrogen
dioxide, and the two predictors are a measure of the fuel-air ratio
$E$, and the compression ratio of the engine $C$. Studying the
interactions between $E$ and $C$ leads the authors of~\cite{vcm} to
consider the model
%
\begin{equation}
\label{eq:vcm}
\E(y_i|E_i,C_i) = \beta_0(E_i) + \beta_1(E_i)\cdot C_i.
\end{equation}
This is a linear model with a different intercept and
slope for each $E_i$, subject to the (implicit) constraint that
the intercept and slope should vary smoothly along the $E_i$'s.
We can fit this using~\eqref{eq:dlasso}, in the
following way: first we discretize the continuous observations
$E_1,\ldots,  E_n$ so that they lie into, say, 25 bins. Our design
matrix $X$ is $88 \times50$, with the first 25 columns modeling the
intercept $\beta_0$ and the last 25 modeling the slope $\beta_1$. The
$i$th row of $X$ is
\[
X_{ij} =
\cases{\displaystyle
1 ,&\quad  if $E_i$ lies in the $j$th bin,\cr\displaystyle
C_i ,&\quad  if $E_i$ lies in the $(j+25)$th bin,\cr\displaystyle
0 ,&\quad  otherwise.
}
\]
Finally, we choose
\[
D = \left[
\matrix{\displaystyle
D_{\mathrm{tf},3} & 0 \cr\displaystyle
0 & D_{\mathrm{tf},3}
}
\right],
\]
where $D_{\mathrm{tf},3}$ is the cubic trend filtering matrix (the
choice $D_{\mathrm{tf},3}$ is not crucial and of course can be
replaced by a higher or lower order trend filtering matrix). The
matrix $D$ is structured in this way so that we penalize the
smoothness of the first 25 and last 25 components of
$\beta=(\beta_0,\beta_1)\T$ individually. With $X$ and $D$ as
described, solving the optimization problem~\eqref{eq:dlasso} gives
the coefficients shown in Figure~\ref{fig:vcm}, which appear quite
similar to the fits in~\cite{vcm}.

%
\begin{figure}

\includegraphics{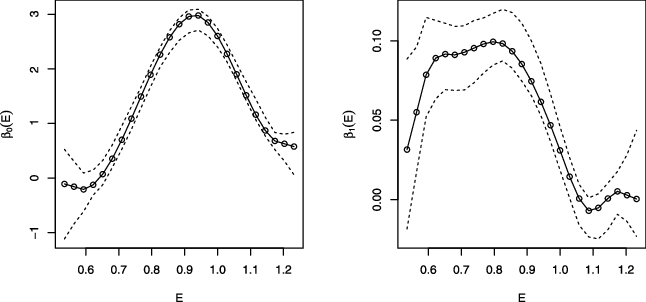}

\caption{The intercept and slope of the varying-coefficient
model \protect\eqref{eq:vcm} for the engine data of \protect\cite{vcm}, fit using
\protect\eqref{eq:dlasso} with a cubic trend filtering penalty matrix (and
$\lambda=3$). The dashed lines show 85\% bootstrap confidence
intervals from 500 bootstrap samples.}
\label{fig:vcm}
\end{figure}

We conclude this section with a generalized lasso application of
\cite{owenoutlie}, in which the penalty is not structurally-based,
unlike the examples discussed previously.
Suppose that we observe $y_1,\ldots,  y_n$, and we believe the majority
of these points follow a linear model $\E(y_i|X_i) = \beta\T X_i$ for
some covariates $X_i=(X_{i1},\ldots,  X_{ip})\T$, except that a small
number of the $y_i$ are outliers and do not come from this model. To
determine which points are outliers, one might consider the
problem
%
\begin{equation}
\label{eq:outlie}
\minimize_{z \in\R^n,   \beta\in\R^p}   \halft\|z-X\beta\|_2^2
 \qquad \mbox{subject to }   \|z-y\|_0 \leq k
\end{equation}
for a fixed integer $k$. Here $\|x\|_0 = \sum_i 1(x_i \not= 0)$. Thus
by setting $k=3$, for example, the solution $\hat{z}$ of
\eqref{eq:outlie} would indicate which 3 points should be considered
outliers, in that $\hat{z}_i \not= y_i$ for exactly 3 coordinates.
A natural convex relaxation of problem~\eqref{eq:outlie} is
%
\begin{equation}
\label{eq:outlie1}
\minimize_{z \in\R^n,   \beta\in\R^p}   \halft\|z-X\beta\|_2^2 +
\lambda\|z-y\|_1,
\end{equation}
where we have also transformed the problem from bound form to Lagrange
form. Letting $\alpha=y-z$, this can be rewritten as
%
\begin{equation}
\label{eq:outlie2}
\minimize_{\alpha\in\R^n,   \beta\in\R^p}   \halft
\|y-\alpha-X\beta\|_2^2 +
\lambda\|\alpha\|_1,
\end{equation}
which fits into the form of problem~\eqref{eq:dlasso}, with design
matrix $\widetilde{X}=\left[ I \enskip X \right]$, coefficient vector
$\tbeta=(\alpha,\beta)\T$, and
penalty matrix $D=\left[ I \enskip 0 \right]$. Figure~\ref{fig:outlie}
shows a simple example with $p=1$.

%
\begin{figure}

\includegraphics{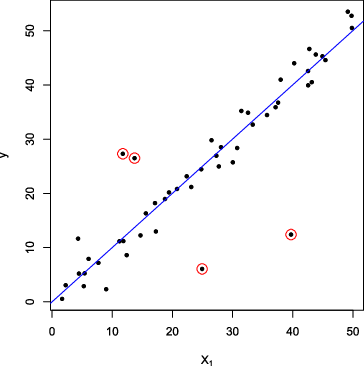}

\caption{A simple example of using problem
\protect\eqref{eq:dlasso} to perform outlier detection. Written in
the form
\protect\eqref{eq:outlie2}, the blue line denotes the fitted slope
$\hbeta$, while the red circles indicate the outliers, as
determined by the coordinates of $\hat\alpha$ that are nonzero
(for $\lambda=8$).}
\label{fig:outlie}
\end{figure}

After reading the examples in this section, a natural question is:
when can a generalized lasso problem~\eqref{eq:dlasso}
be transformed into a regular lasso problem~\eqref{eq:lasso}?
(Recall, e.g., that this is possible for an orthogonal $D$,
as we discussed in the wavelet smoothing example.) We discuss this in
the next section.

\section{When does a generalized lasso problem reduce to a lasso problem?}
\label{sec:lasso}

If $D$ is $p \times p$ and invertible, we can transform
variables in problem~\eqref{eq:dlasso} by $\theta=D\beta$, yielding
the lasso problem
%
\begin{equation}
\label{eq:dlasso2}
\minimize_{\theta\in\R^p}   \halft\|y-XD^{-1}\theta\|^2_2 +
\lambda\|\theta\|_1.
\end{equation}
More generally, if $D$ is $m \times p$ and $\operatorname{rank}(D)=m$ (note
that this necessarily means $m\leq p$), then we can still transform
variables and get a lasso problem. First, we construct a
$p \times p$ matrix
$\tD=\left[
 {D\atop  A}
\right]$
with $\operatorname{rank}(\tD)=p$, by finding a $(p-m)\times p$ matrix $A$
whose rows are orthogonal to those in $D$. Then we change variables to
$\theta=(\theta_1,\theta_2)\T
= \tD\beta$, so that the generalized lasso~\eqref{eq:dlasso} becomes
%
\begin{equation}
\label{eq:dlasso3}
\minimize_{\theta\in\R^p}   \halft\|y-X\tD^{-1}\theta\|^2_2 +
\lambda\|\theta_1\|_1.
\end{equation}
This is almost a regular lasso, except that the $\ell_1$
penalty only covers part of the coefficient vector. First, write
$X\tD^{-1}\theta= X_1\theta_1 + X_2\theta_2$; then, it is clear that
at the solution the second block of the coefficients is given by
a linear regression:
\[
\hat{\theta}_2 = (X_2\T X_2)^{-1} X_2\T(y-X_1\hat{\theta}_1).
\]
Therefore, we can rewrite problem~\eqref{eq:dlasso3} as
%
\begin{equation}
\label{eq:dlasso4}
\minimize_{\theta_1 \in\R^m}   \halft\|(I-P)y-(I-P)X_1\theta_1\|^2_2
+ \lambda\|\theta_1\|_1,
\end{equation}
where $P=X_2\T(X_2\T X_2)^{-1}X_2\T$, the projection onto the column
space of $X_2$. The LARS algorithm provides the solution path of such
a lasso problem~\eqref{eq:dlasso4}, from which we can back-transform
to get the generalized lasso solution:
$\hbeta= \tD^{-1} \hat{\theta}$.

However, if $D$ is $m \times p$ and $\operatorname{rank}(D)<m$, then such a
transformation is not possible, and LARS cannot be used to find the
solution path of the generalized lasso problem
\eqref{eq:dlasso}. Further, in this case, the authors of~\cite{avs}
establish what they call an ``unbridgeable'' gap between problems
\eqref{eq:lasso} and~\eqref{eq:dlasso}, based on the
geometric properties of their solutions.

%
\begin{table}[b]
\tabcolsep=0pt
\caption{Examples from Section \protect\ref{sec:apps} that fall into
the cases $\operatorname{rank}(D)=m$ and $\operatorname{rank}(D)<m$}
\label{table:rank}
\begin{tabular*}{\textwidth}{@{\extracolsep{\fill}}p{177pt}p{175pt}@{}}
\hline
$\bolds{\operatorname{rank}(D)=m}$ &\multicolumn{1}{l@{}}{$\bolds{\operatorname{rank}(D)<m}$}\\
\hline
  \mbox{$\bullet$ The 1d fused lasso} \mbox{$\bullet$ Polynomial trend filtering of any order} \mbox{$\bullet$ Wavelet smoothing with an orthogonal} \mbox{\hphantom{$\bullet$ }wavelet basis}\break  \mbox{$\bullet$ Outlier detection}
  & \mbox{$\bullet$ The fused lasso on any graph that has} \mbox{\hphantom{$\bullet$ }more edges $m$ than nodes $p$ (e.g., the 2d} \mbox{\hphantom{$\bullet$ }fused lasso)}\break  \mbox{$\bullet$ The sparse fused lasso on any graph} \mbox{$\bullet$ Wavelet smoothing with an overcomplete}\break  \mbox{\hphantom{$\bullet$ }wavelet set}
 \\
  \hline
\end{tabular*}
\end{table}

While several of the examples from Section~\ref{sec:apps} satisfy
$\operatorname{rank}(D)=m$, and hence admit a lasso transformation, a good
number also fall into the case $\operatorname{rank}(D)<m$, and suggest the
need for a novel path algorithm. These are summarized in Table
\ref{table:rank}. Therefore, in the next section, we derive the
Lagrange dual of problem~\eqref{eq:dlasso}, which leads to a nice
algorithm to compute the solution path of~\eqref{eq:dlasso} for an
arbitrary penalty matrix $D$.

\section{\texorpdfstring{The Lagrange dual problem.}{The Lagrange dual problem}}
\label{sec:dual}

First, we consider the generalized lasso in the signal
approximation case, $X=I$:
%
\begin{equation}
\label{eq:dlassoi}
\minimize_{\beta\in\R^n}   \halft\|y-\beta\|_\ltwo^2 +
\lambda\|D\beta\|_\lone.
\end{equation}
Essentially, problem~\eqref{eq:dlassoi} is difficult to analyze
directly because the nondifferentiable $\ell_1$ penalty is composed
with a linear transformation of $\beta$.
Following an argument of~\cite{l1tf}, we rewrite the problem as
\[
\minimize_{\beta\in\R^n,   z \in\R^m}   \halft\|y-\beta\|
_\ltwo^2
+ \lambda\|z\|_\lone  \qquad  \mbox{subject to }   D\beta=z.
\]
The Lagrangian is hence
\[
\halftt\|y-\beta\|_\ltwo^2 + \lambda\|z\|_\lone+ u\T(D\beta-z),
\]
and to derive the dual problem, we minimize this over
$\beta,z$. The terms involving $\beta$ are just a quadratic, and
up to some constants (not depending on~$u$)\looseness=-1
\[
\min_\beta \biggl(\halft\|y-\beta\|_\ltwo^2 + u\T D\beta \biggr) =
-\halft\|y-D\T u\|_\ltwo^2,
\]\looseness=0
while
\[
\min_z  (\lambda\|z\|_\lone- u\T z ) =
\cases{\displaystyle
0 ,&\quad  if $\|u\|_\linf\leq\lambda$,\cr\displaystyle
-\infty,&\quad  otherwise.
}
\]
Therefore, the dual problem of~\eqref{eq:dlassoi} is
%
\begin{equation}
\label{eq:dual}
\minimize_{u \in\R^m}   \halft\|y-D\T u\|_\ltwo^2  \qquad  \mbox{subject
to }  \|u\|_\linf\leq\lambda.
\end{equation}
Immediately, we can see that~\eqref{eq:dual} has a ``nice''
constraint set, $\{u \dvtx  \|u\|_\linf\leq\lambda\}$, which is simply
a box, free of any linear transformation. It is also important to
note the difference in dimension: the dual problem has a variable $u
\in\R^m$, whereas the original problem~\eqref{eq:dlassoi}, called the
primal problem, has a variable $\beta\in\R^n$.

When $\operatorname{rank}(D)<m$, the dual problem is not strictly
convex, and so it can have many solutions. On the other hand, the
primal problem is always strictly convex and always has a unique
solution.
The primal problem is also strictly feasible (it has no constraints),
and so strong duality holds (see Section 5.2 of~\cite{convex}).
The primal\vspace*{1pt} and dual solutions---written as $\hbeta_\lambda$ and
$\hu_\lambda$, respectively, to emphasize the dependence on
$\lambda$---are related by
%
\begin{equation}
\label{eq:primaldual}
\hbeta_\lambda= y - D\T\hu_\lambda.
\end{equation}
Furthermore, each coordinate $i=1,\ldots,  m$ of the dual solution
satisfies
%
\begin{equation}
\label{eq:dualsign}
\hu_{\lambda,i} \in
\cases{\displaystyle
\{+\lambda\} ,&\quad  if $(D\hbeta_\lambda)_i > 0$,\cr\displaystyle
\{-\lambda\} ,&\quad  if $(D\hbeta_\lambda)_i < 0$,\cr\displaystyle
[-\lambda,\lambda] ,&\quad  if $(D\hbeta_\lambda)_i = 0$.
}\vadjust{\goodbreak}
\end{equation}
This last equation tells us that the dual
coordinates that are equal to $\lambda$ in absolute value,
%
\begin{equation}
\label{eq:bset}
\cB=\{ i \dvtx  |\hu_{\lambda,i}| = \lambda\},
\end{equation}
are the coordinates of $D\hbeta_\lambda$ that are ``allowed''
to be nonzero. But this does necessarily mean that
$(D\hbeta_\lambda)_i \not= 0$ for all $i \in\cB$.

For a general design matrix $X$, we can apply a similar argument
to derive the dual of~\eqref{eq:dlasso}:
%
\begin{eqnarray}
\label{eq:xdual}
\minimize_{u \in\R^m}  \halft(X\T y-D\T u)\T(X\T X)^+ (X\T y - D\T u)
\nonumber
\\[-8pt]
\\[-8pt]
\eqntext{\mbox{subject to }
 \|u\|_\linf\leq\lambda, \ D\T u \in\operatorname{row}(X).}
\end{eqnarray}
This looks complicated, certainly in comparison to problem
\eqref{eq:dual}. However, the inequality constraint on $u$ is still a
simple (untransformed) box. Moreover, we can make~\eqref{eq:xdual}
look like~\eqref{eq:dual} by a change of variables. This will be
discussed later in Section~\ref{sec:x}.

In the next two sections, Sections~\ref{sec:1d} and~\ref{sec:d}, we
restrict our attention to the case $X=I$ and derive an algorithm to
find a solution path of the dual~\eqref{eq:dual}.
This gives the desired primal solution path, using the
relationship~\eqref{eq:primaldual}. Since our focus is on solving
the dual problem, we write simply ``solution'' or ``solution path'' to
refer to the dual versions. Though we will eventually consider an
arbitrary matrix $D$ in Section~\ref{sec:d}, we begin by studying
the 1d fused lasso in Section~\ref{sec:1d}. This case is especially
simple, and we use it to build the framework for the path algorithm in
the general case.

\section{\texorpdfstring{The 1d fused lasso.}{The 1d fused lasso}}
\label{sec:1d}

In this setting, we have $D=D_\mathrm{1d}$, the $(n-1)\times n$ matrix
given in~\eqref{eq:d1d}. Now the dual problem~\eqref{eq:dual}
is strictly convex (since $D_\mathrm{1d}$ has rank equal to its number
of rows), and therefore it has a unique solution. In order to
efficiently compute the solution path, we use a lemma that allows us,
at different stages, to reduce the dimension of the problem by one.

\subsection{\texorpdfstring{The boundary lemma.}{The boundary lemma}}
\label{sec:blem}
Consider the constraint set $\{u\dvtx  \|u\|_\linf\leq\lambda\} \subseteq
\R^{n-1}$: this is a box centered around the origin with side
length $2\lambda$. We say that coordinate $i$ of $u$ is ``on the
boundary'' (of this box) if $|u_i|=\lambda$. For the 1d fused lasso,
it turns out that coordinates of the solution that are on
the boundary will remain on the boundary indefinitely as $\lambda$
decreases.
This idea can be stated more precisely as follows.
\begin{lemma}[(The boundary lemma)]
\label{lemma:boundary} Suppose that $D=D_\mathrm{1d}$, the 1d
fused lasso matrix in~\eqref{eq:d1d}. For any coordinate $i$, the
solution $\hu_\lambda$ of~\eqref{eq:dual} satisfies
\[
\hu_{\lambda_0,i} = \lambda_0   \quad  \Rightarrow \quad
\hu_{\lambda,i} = \lambda \qquad
\mbox{for all }   \lambda\in[0,\lambda_0]
\]
and
\[
\hu_{\lambda_0,i} = -\lambda_0  \quad  \Rightarrow \quad
\hu_{\lambda,i} = -\lambda \qquad
\mbox{for all }   \lambda\in[0,\lambda_0].\vadjust{\goodbreak}
\]
\end{lemma}

The proof is given in~\cite{supp}. It is interesting to note
a connection between the boundary lemma and a lemma of
\cite{pco}, which states that
%
\begin{equation}
\label{eq:fuselem}
\hat{\beta}_{\lambda_0,i} = \hat{\beta}_{\lambda_0,i+1}  \quad  \Rightarrow \quad
\hat{\beta}_{\lambda,i} = \hat{\beta}_{\lambda,i+1}  \qquad  \mbox{for all }
\lambda\geq\lambda_0
\end{equation}
for this same problem. In other words, this lemma says that no two
equal primal coordinates can become unequal with increasing $\lambda$.
In general $|\hu_{\lambda,i}|=\lambda$ is not equivalent to
$(D\hbeta_\lambda)_i\not=0$, but these two statements are
equivalent for the 1d fused lasso problem (see the
primal-dual correspondence in Section~\ref{sec:1dprops}),
and therefore the boundary lemma is equivalent to~\eqref{eq:fuselem}.

\subsection{\texorpdfstring{Path algorithm.}{Path algorithm}}
\label{sec:1dpath}
This section is intended to explain the path algorithm from a
conceptual point of view, and no rigorous arguments for its
correctness are made here. We defer these until Section
\ref{sec:dpath}, when we revisit the problem in the context of a
general matrix $D$.

The boundary lemma describes the behavior of the solution as $\lambda$
decreases, and therefore it is natural to construct the solution path
by moving the parameter from $\lambda=\infty$ to $\lambda=0$.
As will be made apparent from the details of the algorithm,
the solution path is a piecewise linear function of~$\lambda$, with a
change in slope occurring whenever one of its coordinate paths hits the
boundary.
The key observation is that, by the boundary lemma, if a coordinate
hits the boundary it will stay on the boundary for the rest of
the path down to $\lambda=0$. Hence, when it hits the boundary we can
essentially eliminate a coordinate from consideration (since we know
its value at each smaller $\lambda$), recompute the slopes of the
other coordinate paths, and move until another coordinate hits the
boundary.

As we construct the path, we maintain two lists:
$\cB=\cB(\lambda)$, which contains the coordinates that are currently
on the boundary; and $s=s(\lambda)$, which contains their signs.
For example, if we have $\cB(\lambda)=(5,2)$ and $s(\lambda)=(-1,1)$,
then this means that $\hu_{\lambda,5}=-\lambda$ and
$\hu_{\lambda,2}=\lambda$.
We call the coordinates in $\cB$ the ``boundary coordinates,'' and the
rest the ``interior coordinates.''
Now we can describe the algorithm:

\begin{algorithm}[(Dual path algorithm for the 1d fused lasso)]
\label{alg:1d}
\smallskipamount=0pt
\begin{itemize}
\item  Start with $\lambda_0=\infty$,
$\cB=\varnothing$  and $s=\varnothing$.
\item  For $k=0,\ldots,  n-2$:
\begin{enumerate}[3.]
\item[1.] Compute the solution at $\lambda_k$ by least squares, as in
\eqref{eq:ls}.
\item[2.] Continuing in a linear direction from the solution, compute
$\lambda_{k+1}$, when an interior coordinate will next hit the
boundary, as in~\eqref{eq:hittime} and~\eqref{eq:nextlam}.
\item[3.] Add this coordinate to $\cB$ and its sign to $s$.
\end{enumerate}
\end{itemize}
\end{algorithm}

The algorithm's details appear slightly more complicated, but
this is only because of notation. If $\cB=(i_1,\ldots,  i_k)$, then we
define for a matrix $A$ and a vector~$x$
\[
A_\cB= \left[
\matrix{
A_{i_1} \cr \vdots\cr A_{i_k}
}
\right]  \quad  \mbox{and} \quad
x_\cB= (x_{i_1},\ldots,  x_{i_k})\T,
\]
where $A_i$ is the $i$th row of $A$. In words: $A_\cB$ indexes
the rows of $A$ that are in $\cB$, and $x_\cB$ indexes the coordinates
of $x$ in $\cB$. We use the subscript $-\cB$, as
in $A_{-\cB}$ or $x_{-\cB}$, to index over all rows or coordinates
except those in $\cB$. Note that $\cB$ as defined above (in the
paragraph preceding the algorithm)
is consistent with our previous definition~\eqref{eq:bset}, except
that here we treat $\cB$ as an ordered list instead
of a set (its ordering only needs to be consistent with that of
$s$). Also, we treat $s$ as a vector when convenient.

When $\lambda=\infty$, the problem is unconstrained, and so clearly
$\cB=\varnothing$ and $s=\varnothing$.
But more generally, suppose that we are at the $k$th
iteration, with boundary set $\cB=\cB(\lambda_k)$
and signs $s=s(\lambda_k)$. By the boundary lemma, the solution
satisfies
\[
\hu_{\lambda,\cB} = \lambda s  \qquad
\mbox{for all }   \lambda\in[0,\lambda_k].
\]
Therefore, for $\lambda\leq\lambda_k$, we can reduce the optimization
problem~\eqref{eq:dual} to
%
\begin{equation}
\label{eq:dualrewrite}
\minimize_{u_{-\cB}}  \halft\|y - \lambda(D_\cB)\T s
- (D_{-\cB})\T u_{-\cB} \|_\ltwo^2  \qquad  \mbox{subject to }
\|u_{-\cB}\|_\linf\leq\lambda,\hspace*{-35pt}
\end{equation}
which involves solving for just the interior coordinates. By
construction, $\hu_{\lambda_k,-\cB}$ lies strictly between $-\lambda_k$
and $\lambda_k$ in every coordinate. Therefore, it is found by
simply minimizing the objective function in~\eqref{eq:dualrewrite},
which gives the least squares estimate
%
\begin{equation}
\label{eq:ls}
\hu_{\lambda_k,-\cB} =  ( D_{-\cB}(D_{-\cB})\T )^{-1}
D_{-\cB}
 \bigl(y-\lambda_k (D_\cB)\T s  \bigr).
\end{equation}
Let $a-\lambda_k b$ denote the right-hand side above.
For $\lambda\leq\lambda_k$, the interior solution will
continue to be $\hu_{\lambda,-\cB}=a-\lambda b$ until one of its
coordinates hits the boundary. This critical value is determined by
solving, for each $i$, the equation $a_i-\lambda b_i = \pm\lambda$;
a~simple calculation shows that the solution is
%
\begin{equation}
\label{eq:hittime}
t_i = \frac{a_i}{b_i \pm1} =
\frac{ [ (D_{-\cB}(D_{-\cB})\T )^{-1}D_{-\cB} y ]_i}
{ [ (D_{-\cB}(D_{-\cB})\T )^{-1}D_{-\cB}(D_\cB)\T
s ]_i
\pm1}
\end{equation}
(only one of $+1$ or $-1$ will yield a value $t_i \in
[0,\lambda_k]$), which we call the ``hitting time'' of coordinate
$i$. We take $\lambda_{k+1}$ to be maximum of these hitting times
%
\begin{equation}
\label{eq:nextlam}
\lambda_{k+1} = \max_i   t_i.
\end{equation}
Then we compute
\[
i_{k+1} = \argmax_i   t_i  \quad  \mbox{and}  \quad
s_{k+1} = \sign (\hu_{\lambda_{k+1},i_{k+1}} ),
\]
and append $i_{k+1}$ and $s_{k+1}$ to $\cB$ and $s$,
respectively.

\subsection{\texorpdfstring{Properties of the solution path.}{Properties of the solution path}}
\label{sec:1dprops}
Here, we study some of the path's basic properties. Again we defer any
rigorous arguments until Section~\ref{sec:dprops}, when we consider a
general penalty matrix $D$. Instead, we demonstrate them by way of a
simple example.

Consider Figure~\ref{fig:paths}(a), which shows the coordinate paths
$\hu_{\lambda,i}$ for an example with $n=8$.
Recall that it is natural to interpret the paths from right to
left ($\lambda=\infty$ to $\lambda=0$).
Initially all of the slopes are zero, because when
$\lambda=\infty$ the solution is just the least squares estimate
$(DD\T)^{-1}Dy$, which has no dependence on $\lambda$. When a
coordinate path first hits the boundary (the topmost path, drawn in
red) the slopes of the other
paths change, and they do not change again until another coordinate
hits the boundary (the bottommost path, drawn in green), and so on,
until all coordinates are on the boundary.

%
\begin{figure}

\includegraphics{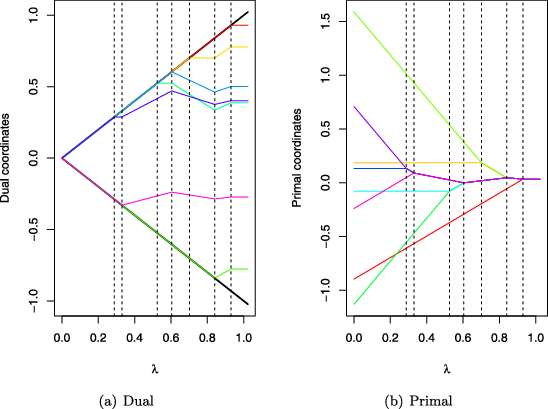}

\caption{(\textup{a}) Dual and (\textup{b}) primal coordinate paths for a small
problem with $n=8$.}
\label{fig:paths}
\end{figure}

The picture suggests that the path $\hu_\lambda$ is continuous and
piecewise linear with respect to $\lambda$, with changes in slope or
``kinks'' at the values $\lambda_1, \ldots, \lambda_{n-1}$ visited by
the algorithm. (Piecewise linearity is obvious from the algorithm's
construction of the path, but continuity is not.)
This is also true in the general $D$ case, although the solution
path can have more than $m$ kinks for an $m\times n$ matrix $D$.

On the other hand, Figure~\ref{fig:paths}(b) shows the
corresponding primal coordinate paths
\[
\hbeta_{\lambda,i} = (y-D\T\hu_\lambda)_i.
\]
As $\hu_\lambda$ is a continuous piecewise linear function of
$\lambda$, so is $\hbeta_\lambda$, again with kinks at $\lambda_1,
\ldots,\lambda_{n-1}$. In contrast to the dual versions, it is natural
to interpret the primal coordinate paths from left to right, because
in this direction the coordinate paths become adjoined, or ``fused,''
at a some value of $\lambda$. The primal picture suggests that these
fusion values are the same as the kinks $\lambda_1,\ldots,
\lambda_{n-1}$, that is:
\begin{itemize}
\item%
\textit{Primal-dual correspondence for the 1d fused lasso.}
The values of $\lambda$ at which two primal coordinates fuse are
exactly the values of $\lambda$ at which a dual coordinate hits the
boundary.
\end{itemize}
A similar property holds for the fused lasso on an arbitrary
graph, although the primal-dual correspondence is a
little more complicated for this case.

Note that as $\lambda$ decreases
in Figure~\ref{fig:paths}(a), no dual coordinate paths leave the
boundary. This is prescribed by the boundary lemma. As $\lambda$
increases in Figure~\ref{fig:paths}(b), no primal two coordinates
split apart, or ``unfuse.'' This is prescribed by a lemma of
\cite{pco} that we paraphrased in~\eqref{eq:fuselem}, and the two
lemmas are equivalent.

\section{\texorpdfstring{A general penalty matrix $D$.}{A general penalty matrix $D$}}
\label{sec:d}

Now we consider~\eqref{eq:dual} for general $m\times n$ matrix
$D$. The first question that comes to mind is: does the boundary lemma
still hold? If $DD\T$ is diagonally dominant, that is
%
\begin{equation}
\label{eq:dd}
(DD\T)_{ii} \geq\sum_{j\not=i} |(DD\T)_{ij}|
 \qquad  \mbox{for }   i=1,\ldots, m,
\end{equation}
then indeed the boundary lemma is still true. (See~\cite{supp}.)
Therefore, the path algorithm for such a $D$ is the same as that
presented in the previous section.

It is easy to check the 1d fused lasso matrix is diagonally dominant,
as both the left- and right-hand sides of the inequality in
\eqref{eq:dd} are equal to 2 when $D=D_\mathrm{1d}$. Unfortunately,
neither the 2d fused lasso matrix nor any of the trend filtering
matrices satisfy condition~\eqref{eq:dd}.
In fact, examples show that the boundary lemma does not
hold for these cases. However, inspired by the 1d fused lasso, we
can develop a similar strategy to compute the full solution path for
an arbitrary matrix $D$. The difference is: in addition to checking
when coordinates will hit the boundary, we have to check when
coordinates will leave the boundary as well.

\subsection{\texorpdfstring{Path algorithm.}{Path algorithm}}
\label{sec:dpath}
Recall that we defined, at a particular
$\lambda_k$, the ``hitting time'' of an interior coordinate path to
the value of $\lambda\leq\lambda_k$ at which this path hits the
boundary. Similarly, let us define the ``leaving time'' of a
boundary coordinate path to be the value of $\lambda\leq\lambda_k$
at which this path leaves the boundary (we will make this idea more
precise shortly). We call the coordinate with the largest
hitting time the ``hitting coordinate,'' and the one with the largest
leaving time the ``leaving coordinate.'' As before, we maintain a
list $\cB$ of boundary coordinates, and $s$ contains their signs.
The algorithm for a general matrix $D$ is:

\begin{algorithm}[(Dual path algorithm for a general $D$)]
\label{alg:d}
\begin{itemize}
\item Start with $k=0$, $\lambda_0=\infty$, $\cB=\varnothing$, and
$s=\varnothing$.
\item While $\lambda_k>0$:
\begin{enumerate}[3.]
\item Compute a solution at $\lambda_k$ by least squares, as in
\eqref{eq:dls}.
\item Compute the next hitting time $h_{k+1}$, as in
\eqref{eq:dhittime} and~\eqref{eq:dnextht}.
\item Compute the next leaving time $l_{k+1}$, as in
\eqref{eq:dpreleave},~\eqref{eq:dleavetime}  and~\eqref{eq:dnextlt}.
\item Set $\lambda_{k+1}=\max\{h_{k+1},l_{k+1}\}$. If
$h_{k+1} > l_{k+1}$, then add the hitting coordinate to $\cB$ and its
sign $s$, otherwise remove the leaving coordinate to $\cB$ and its
sign from $s$. Set $k=k+1$.
\end{enumerate}
\end{itemize}
\end{algorithm}

Although the intuition for this algorithm comes from the
1d fused lasso problem, its details are derived from a more technical
point of view, via the Karush--Kuhn--Tucker (KKT) optimality
conditions. For our problem~\eqref{eq:dual}, the KKT conditions are
%
\begin{equation}
\label{eq:kkt}
(DD\T u)_i - (Dy)_i + \alpha\gamma_i = 0  \qquad  \mbox{for }
i=1,\ldots, m,
\end{equation}
where $u,\alpha,\gamma$ are subject to the constraints
\begin{subequations}
\begin{eqnarray}
\label{eq:a}
\|u\|_\linf&\leq&\lambda,\\
\label{eq:b}
\alpha&\geq&0, \\
\label{eq:c}
\alpha\cdot (\|u\|_\linf-\lambda ) &=& 0,
\\
\label{eq:d}
\|\gamma\|_\lone &\leq& 1,\\\label{eq:e}
\gamma\T u &=& \|u\|_\linf.
\end{eqnarray}
\end{subequations}
\noindent
Constraints~\eqref{eq:d} and~\eqref{eq:e} say that $\gamma$ must be a
subgradient of $\|u\|_\linf$ with respect to $u$.
Subgradients are a generalization of gradients to the case of
nondifferentiable functions---for an overview, see~\cite{bert}.

A necessary and sufficient condition for $u$ to be a solution to
\eqref{eq:dual} is that $u,\alpha,\gamma$ satisfy
\eqref{eq:kkt} and (\ref{eq:a})--(\ref{eq:e}) for some $\alpha$ and $\gamma$.
The basic idea is that hitting times are events in which~\eqref{eq:a}
is violated, and leaving times are events in which (\ref{eq:b})--(\ref{eq:e}) are
violated.
We describe what happens at the $k$th iteration. At
$\lambda=\lambda_k$, the solution is given by
$\hu_{\lambda_k,\cB}=\lambda_k s$ for the boundary coordinates and the
least squares estimate
%
\begin{equation}
\label{eq:dls}
\hu_{\lambda_k,-\cB} =  ( D_{-\cB}(D_{-\cB})\T )^+
D_{-\cB}
 \bigl(y-\lambda_k (D_\cB)\T s  \bigr)
\end{equation}
for the interior coordinates. Here $A^+$ denotes the
(Moore--Penrose) pseudoinverse of a matrix $A$, which is needed as
$D$ may not have full row rank. Write $\hu_{\lambda_k,-\cB} =
a-\lambda_k
b$. Like the 1d fused lasso case, we decrease $\lambda$ and continue
in a linear direction from the interior solution at $\lambda_k$,
proposing $\hu_{\lambda,-\cB}=a-\lambda b$.\vadjust{\goodbreak} We first determine when
a coordinate of $a-\lambda b$ will hit the boundary. The same
calculation as before gives the hitting times
%
\begin{equation}
\label{eq:dhittime}
t^{(\mathrm{hit})}_i = \frac{a_i}{b_i \pm1} =
\frac{ [ (D_{-\cB}(D_{-\cB})\T )^+D_{-\cB} y ]_i}
{ [ (D_{-\cB}(D_{-\cB})\T )^+D_{-\cB}(D_\cB)\T s  ]_i
\pm1}.
\end{equation}
(Only one of $+1$ or $-1$ will yield a value in
$[0,\lambda_k]$.) Hence, the next hitting time is
%
\begin{equation}
\label{eq:dnextht}
h_{k+1} = \max_i   t^{(\mathrm{hit})}_i.
\end{equation}

The new step is to determine when a boundary coordinate will next
leave the boundary. After examining the constraints (\ref{eq:b})--(\ref{eq:d}), we
can express the leaving time of the $i$th boundary coordinate by first
defining
%
\begin{eqnarray}
\label{eq:dpreleave}
c_i &=& s_i \cdot \bigl[ D_\cB [I-(D_{-\cB})\T
 (D_{-\cB}(D_{-\cB})\T )^+ D_{-\cB} ]y  \bigr]_i,
\nonumber
\\[-8pt]
\\[-8pt]
d_i &=& s_i \cdot \bigl[ D_\cB [I-(D_{-\cB})\T
 (D_{-\cB}(D_{-\cB})\T )^+ D_{-\cB} ] (D_\cB)\T s
 \bigr]_i,
\nonumber
\end{eqnarray}
and then the leaving time is
%
\begin{equation}
\label{eq:dleavetime}
t^{(\mathrm{leave})}_i =
\cases{\displaystyle
c_i/d_i ,&\quad  if    $c_i<0$    and   $d_i<0 $,\cr\displaystyle
0 ,&\quad  otherwise.
}
\end{equation}
Therefore, the next leaving time is
%
\begin{equation}
\label{eq:dnextlt}
l_{k+1} = \max_i   t^{(\mathrm{leave})}_i.
\end{equation}

The last step of the iteration moves until the next critical
event---hitting time or leaving time, whichever happens first.
We can verify that the path visited by the
algorithm satisfies the KKT conditions~\eqref{eq:kkt} and
(\ref{eq:a})--(\ref{eq:e}) at each $\lambda$, and hence is indeed a
solution path of the dual problem~\eqref{eq:dual}.
This argument, as well a derivation of the leaving
times given in~\eqref{eq:dpreleave} and~\eqref{eq:dleavetime}, can be
found in~\cite{supp}.

\subsection{\texorpdfstring{Properties of the solution path.}{Properties of the solution path}}
\label{sec:dprops}
Suppose that the algorithm terminates after $T$ iterations.
By construction, the returned solution path $\hu_\lambda$
is piecewise linear with respect to $\lambda$, with kinks at
$\lambda_1,\ldots,\lambda_T$. Continuity, on the other hand, is
a little more subtle: because of the specific choice of the
pseudoinverse solution in~\eqref{eq:dls}, the path
$\hu_\lambda$ is also continuous over $\lambda$.
[When $A$ does not have full column rank, there are many minimizers of
$\|z-Ax\|_\ltwo$, and $x=(A\T A)^+ A\T z$ is only one of them.] The
proof of continuity appears in~\cite{supp}.

Since the primal solution path $\hbeta_\lambda$ can be
recovered from $\hu_\lambda$ by the linear transformation
\eqref{eq:primaldual}, the path $\hbeta_\lambda$ is also continuous
and piecewise linear in~$\lambda$. The kinks in this path are
necessarily a subset of $\{\lambda_1,\ldots,\lambda_T\}$.
However, this could be a strict inclusion as $\operatorname{rank}(D)$
could be ${<}m$, that is, $D\T$ could have a nontrivial null space.
So when does the primal solution path change slope? To answer
this question, it helps to write the solutions in a more explicit
form.\vadjust{\goodbreak}

For any given $\lambda$, let $\cB=\cB(\lambda)$ and $s=s(\lambda)$ be
the current boundary coordinates and their signs. Then we know that
the dual solution can be written as
\begin{eqnarray*}
\hu_{\lambda,\cB} &=& \lambda s ,\\
\hu_{\lambda,-\cB} &= & ( D_{-\cB}(D_{-\cB})\T )^+
D_{-\cB} \bigl(y-\lambda(D_\cB)\T s  \bigr).
\end{eqnarray*}
This means that the dual fit $D\T\hu_\lambda$ is just
%
\begin{eqnarray}
\label{eq:dualfit}
D\T\hu_\lambda&=& (D_\cB)\T\hu_{\lambda,\cB} +
(D_{-\cB})\T\hu_{\lambda,-\cB}\nonumber
\\[-8pt]
\\[-8pt] &=&
\lambda(D_\cB)\T s +
P_{\operatorname{row}(D_{-\cB})} \bigl(y-\lambda(D_\cB)\T s  \bigr),
\nonumber
\end{eqnarray}
where $P_M$ denotes the projection operator onto a linear
subspace $M$ (here the row space of $D_{-\cB}$). Therefore, applying
\eqref{eq:primaldual}, the primal solution is given by
%
\begin{equation}
\label{eq:dlassosol2}
\hbeta_\lambda=  \bigl(I -
P_{\operatorname{row}(D_{-\cB})} \bigr)\bigl (y-\lambda(D_\cB)\T s  \bigr)
=
P_{\operatorname{null}(D_{-\cB})}  \bigl(y-\lambda(D_\cB)\T s  \bigr).
\end{equation}

Equation~\eqref{eq:dlassosol2} is useful for several
reasons. Later, in Section~\ref{sec:df}, we use it along with a
geometric argument to prove a result on the degrees of freedom of
$\hbeta_\lambda$.
But first, equation~\eqref{eq:dlassosol2} can be used to answer our
immediate question about the primal path's changes in slope: it turns
out that $\hbeta_\lambda$ changes slope at $\lambda_{k+1}$ if
$\operatorname{null}(D_{-\cB(\lambda_k)}) \not=
\operatorname{null}(D_{-\cB(\lambda_{k+1})})$, that is, the null
space of $D_{-\cB}$ changes from iterations $k$ to $k+1$. (The proof
is left to~\cite{supp}.) Thus
we have achieved a generalization of the primal-dual correspondence
of Section~\ref{sec:1dprops}:
\begin{itemize}
\item\textit{Primal-dual correspondence for a general $D$.}
The values of $\lambda$ at which at which the primal coordinates
changes slope are the values of $\lambda$ at which the null space of
$D_{-\cB(\lambda)}$ changes.
\end{itemize}

For various applications, the null space of $D_{-\cB}$ can have a nice
interpretation. We present the case for the fused lasso on an
arbitrary graph $\cG$, with $m$ edges and $n$ nodes. We assume without
a loss of generality that $\cG$ is connected (otherwise the problem
decouples into smaller fused lasso problems).
Recall that in this setting each row of $D$ gives the difference
between two nodes connected by an edge. Hence, the null space of $D$ is
spanned by the vector of all ones
\[
\mathbh{1} = (1,1,\ldots,1)\T\in\R^n.
\]
Furthermore, removing a subset of the rows, as in $D_{-\cB}$, is like
removing the corresponding subset of edges, yielding a subgraph
$\cG_{-\cB}$.
It is not hard to see that the dimension of the null space of
$D_{-\cB}$ is equal to the number of connected components in
$\cG_{-\cB}$. In fact, if $\cG_{-\cB}$ has connected components
$A_1,\ldots, A_k$, then the null space of $D_{-\cB}$ is spanned by
$\mathbh{1}_{A_1},\ldots,\mathbh{1}_{A_k} \in\R^m$, the indicator
vectors on these components, that is,
\[
(\mathbh{1}_{A_i})_j = 1(\mbox{node } j \in A_i)  \qquad  \mbox{for }
  j=1,\ldots, n.
\]

When $\cG_{-\cB}$ has connected components $A_1,\ldots, A_k$,
the projection $P_{\operatorname{null}(D_{-\cB})}$ performs a coordinate-wise
average within each group $A_i$:
\[
P_{\operatorname{null}(D_{-\cB})}(x) = \sum_{i=1}^k
\biggl (\frac{(\mathbh{1}_{A_i}) \T x}{|A_i|} \biggr) \cdot
\mathbh{1}_{A_i}.
\]
Therefore, recalling~\eqref{eq:dlassosol2}, we see that coordinates of
the primal solution $\hbeta_\lambda$ are constant (or in other words,
fused) on each group $A_i$.

As $\lambda$ decreases, the boundary set $\cB$ can both
grow and shrink in size; this corresponds to adding an edge to and
removing an edge from the graph $\cG_{-\cB}$, respectively. Since the
null space of $D_{-\cB}$ can only change when $\cG_{-\cB}$
undergoes a change in connectivity, the general primal-dual
correspondence stated above becomes:
\begin{itemize}
\item\textit{Primal-dual correspondence for the fused lasso on a graph.}
In two parts:
\begin{longlist}[(ii)]
\item[(i)] the values of $\lambda$ at which two primal coordinate
groups fuse are the values of $\lambda$ at which a dual
coordinate hits the boundary and disconnects the graph
$\cG_{-\cB(\lambda)}$;
\item[(ii)] the values of $\lambda$ at which two primal coordinate
groups unfuse are the values of $\lambda$ at which a dual
coordinate leaves the boundary and reconnects the graph
$\cG_{-\cB(\lambda)}$.
\end{longlist}
\end{itemize}

Figure~\ref{fig:Dpaths} illustrates this correspondence for a
graph with $n=6$ nodes and $m=9$ edges. Note that the
primal-dual correspondence for the fused lasso on a graph, as stated
above, is consistent with that given in Section
\ref{sec:1dprops}. This is because the 1d fused lasso corresponds to a
chain graph, so removing an edge always disconnects the
graph, and furthermore, no dual coordinates ever leave the boundary by
the boundary lemma.

%
\begin{figure}[t!]

\includegraphics{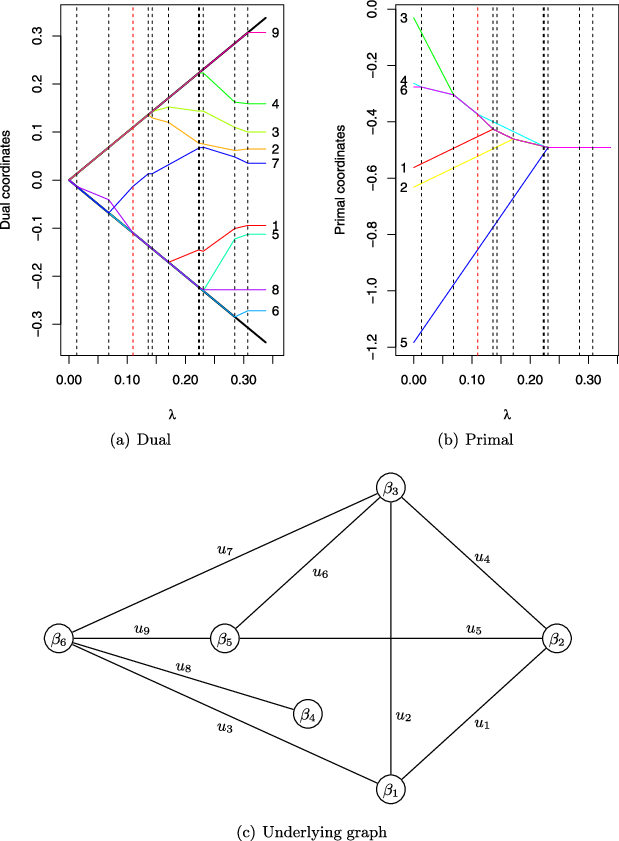}

\caption{(\textup{a}) Dual and (\textup{b}) primal coordinate paths for the
fused lasso applied to the  graph structure shown in (\textup{c}). As $\lambda$
decreases, the first dual coordinate to hit the boundary is
$u_9$, but removing the corresponding edge does not disconnect the
graph, so nothing happens in the primal setting. Then $u_6$ hits the
boundary, and again, removing its edge does not affect the graph's
connectivity, so nothing happens. But when $u_5$ hits the boundary
next, removing its edge disconnects the graph (the node marked
$\beta_5$ becomes its own connected component), and hence two primal
coordinate paths fuse. Note that $u_8$ leaves the boundary at some
point (the red dashed vertical line). Adding its edge reconnects the
graph, and therefore two primal coordinates unfuse.}
\label{fig:Dpaths}
\end{figure}

\section{\texorpdfstring{A general design matrix $X$.}{A general design matrix $X$}}
\label{sec:x}

In the last two sections, we focused on the signal approximation case
$X=I$. In this section, we consider the problem~\eqref{eq:dlasso} when
$X$ is a general $n\times p$ matrix of covariates (and $D$ is a
general $m \times p$ penalty matrix). Our strategy is to again solve
the equivalent dual problem~\eqref{eq:xdual}. At first glance, this
problem looks much more difficult than the dual~\eqref{eq:dual} when
$X=I$. Moreover, the relationship between the primal and dual
solutions is now
%
\begin{equation}
\label{eq:xprimaldual}
\hbeta_\lambda= (X\T X)^+ (X\T y - D\T\hu_\lambda),
\end{equation}
which is also more complicated.

However, suppose that we define $\ty=XX^+ y$ and $\tD=DX^+$, where the
pseudoinverse of the (rectangular) matrix $X$ is $X^+=(X\T
X)^+X\T$. Abbreviating $P=P_{\mathrm{col}(X)}=XX^+$, the objective
function in~\eqref{eq:xdual} becomes
\begin{eqnarray*}
(X\T y-D\T u)\T(X\T X)^+ (X\T y - D\T u)
&=& y\T Py - 2y\T\tD\T u + u\T\tD\tD\T u \\
&=& (y-\tD\T u)\T P (y -\tD\T u) \\
&=& (y-\tD\T u)\T P^2 (y -\tD\T u) \\
&=& (\ty- \tD\T u)\T(\ty- \tD\T u).
\end{eqnarray*}
The first equality above is by the definition of $D$; the second holds
because $P\tD\T=\tD\T$; the third is because $P$ is idempotent; and
the fourth is again due to the identity $P\tD\T=\tD\T$. Therefore we
can rewrite the dual problem~\eqref{eq:xdual} in terms of our
transformed data and penalty matrix:
%
\begin{eqnarray}
\label{eq:xdual2}
\minimize_{u \in\R^m}  \half\|\ty-\tD\T u\|_\ltwo^2
\nonumber
\\[-8pt]
\\[-8pt]
\eqntext{\mbox{subject to }
 \|u\|_\linf\leq\lambda, \  D\T u \in\operatorname{row}(X).}
\end{eqnarray}
It is also helpful to rewrite the relationship~\eqref{eq:xprimaldual}
in terms of our new variables:
%
\begin{equation}
\label{eq:xprimaldual2}
\hbeta_\lambda= X^+(\ty- \tD\T\hu_\lambda),
\end{equation}
which implies that the fit is simply
%
\begin{equation}
\label{eq:xprimaldual3}
X\hbeta_\lambda= \ty- \tD\T\hu_\lambda.
\end{equation}

Modulo the row space constraint, $D\T u \in\operatorname{row}(X)$,
problem~\eqref{eq:xdual2} has exactly the same form as the dual
\eqref{eq:dual} studied in Section~\ref{sec:d}.
In the case that $X$ has full column rank, this extra constraint has
no effect, so we can treat the problem just as before. We discuss this
next.

\subsection{\texorpdfstring{The case $\operatorname{rank}(X)=p$.}{The case $\operatorname{rank}(X)=p$}}
\label{sec:xfull}

Suppose that $\operatorname{rank}(X)=p$, so
$\operatorname{row}(X)=\R^p$ (note that this necessarily means
$p \leq n$). Then the constraint $D\T u \in\operatorname{row}(X)$ is
trivially satisfied for any $u$, and problem~\eqref{eq:xdual2} is
the same as problem~\eqref{eq:dual} that we solved in
Section~\ref{sec:d}, except with $y,D$ replaced by $\ty,\tD$,
respectively. Therefore, we can apply Algorithm~\ref{alg:d} to find a
dual solution path $\hu_\lambda$, which gives the primal solution path
using~\eqref{eq:xprimaldual2}, or the fit using
\eqref{eq:xprimaldual3}.

Fortunately, all of the properties in Section~\ref{sec:dprops}
apply to the current setting as well.
First, we know that the constructed dual path $\hu_\lambda$ is
continuous and piecewise linear, because we are using the same
algorithm as before. This means that $\hbeta_\lambda$ is
also continuous and piecewise linear, since it is given by the linear
transformation~\eqref{eq:xprimaldual2}. Next, we can
follow the same logic in writing out the dual fit $\tD\T\hu_\lambda$
to conclude that
%
\begin{equation}
\label{eq:dlassosol3}
\hbeta_\lambda= X^+ P_{\operatorname{null}(\tD_{-\cB})}
 \bigl(\ty- \lambda(\tD_{-\cB})\T s \bigr)
\end{equation}
or
%
\begin{equation}
\label{eq:dlassosol4}
X\hbeta_\lambda= P_{\operatorname{null}(\tD_{-\cB})}
 \bigl(\ty- \lambda(\tD_{-\cB})\T s \bigr).
\end{equation}
Hence, $0=\tD_{-\cB}X\hbeta_\lambda=D_{-\cB}\hbeta_\lambda$, which
means that $\hbeta_\lambda\in\operatorname{null}(D_{-\cB})$, as before.

Though working with equations~\eqref{eq:dlassosol3} and
\eqref{eq:dlassosol4} may seem complicated (as one would need to
expand the newly defined variables $\ty,\tD$ in terms of $y,D$), it is
straightforward to show that the general primal-dual correspondence
still holds here. This is given in~\cite{supp}.
That is: the primal path $\hbeta_\lambda$ changes
slope at the values of $\lambda$ at which the null space of
$D_{-\cB(\lambda)}$ changes.
For the fused lasso on a graph $\cG$, we indeed still get
fused groups of coordinates in the primal solution, since
$\hbeta_\lambda\in\operatorname{null}(D_{-\cB})$ implies that
$\hbeta_\lambda$ is fused on the connected components of
$\cG_{-\cB}$. Therefore, fusions still correspond to dual coordinates
hitting the boundary and disconnecting the graph, and unfusions still
correspond to dual coordinates leaving the boundary and reconnecting
the graph.

\subsection{\texorpdfstring{The case $\operatorname{rank}(X)<p$.}{The case $\operatorname{rank}(X)<p$}}

If $\operatorname{rank}(X)<p$, then $\operatorname{row}(X)$ is a strict subspace
of $\R^p$. One easy way to avoid dealing with the constraint
$D\T u \in\operatorname{row}(X)$ of~\eqref{eq:xdual2} is to add an $\ell_2$
penalty to our original problem. That is, we consider for a fixed
$\epsilon>0$
%
\begin{equation}
\label{eq:dlassoridge}
\minimize_{\beta\in\R^p}   \halft\|y-X\beta\|_\ltwo^2 +
\lambda\|D\beta\|_\lone+ \epsilon\|\beta\|_\ltwo^2,
\end{equation}
which is the same as
\[
\minimize_\beta  \halft\|y^*-(X^*)\beta\|_\ltwo^2 +
\lambda\|D\beta\|_\lone,
\]
where $y^*=(y,0)\T$ and
$X^*=\left[
 { X \atop \epsilon\cdot I}
\right]$.
Since $\operatorname{rank}(X^*)=p$, we can use\vspace*{1pt} the
strategy discussed in the last section, which is
just applying Algorithm~\ref{alg:d} to a transformed problem, to find
the solution path of~\eqref{eq:dlassoridge}.
Putting aside computational concerns,
it may still be preferable to study problem
\eqref{eq:dlassoridge} instead of problem~\eqref{eq:dlasso}.
Some reasons are:
\begin{itemize}
\item as $\operatorname{rank}(X)<p$, the problem~\eqref{eq:dlasso}
is no longer strictly convex and may not have a unique
solution; this complicates the idea of a solution path, which can now
be discontinuous with respect to $\lambda$
(see~\cite{holger} for a related example in the fused lasso case);
\item the solution of~\eqref{eq:dlassoridge} may actually outperform
that of~\eqref{eq:dlasso} in terms prediction error, analogous to
the advantage of the \textit{elastic net} over the lasso
(see~\cite{enet}).
\end{itemize}

Though adding an $\ell_2$ penalty is easier and, as we suggested,
perhaps even desirable, we can still solve the unmodified problem
\eqref{eq:dlasso} in the $\operatorname{rank}(X)<p$ case, by looking at its
dual~\eqref{eq:xdual2}. We only give a rough sketch of the path
algorithm because in the present setting the solution and its
computation are more complicated.

We can rewrite the row space constraint in~\eqref{eq:xdual2} as $D\T u
\perp\operatorname{null}(X)$. Using the SVD of $X$, we can construct an
orthogonal basis for the null space of~$X$. Let $W$ be the matrix that
has these basis elements in its columns. Then problem
\eqref{eq:xdual2} is now
%
\begin{eqnarray}
\label{eq:xdual3}
\minimize_{u \in\R^m}  \halft\|\ty-\tD\T u\|_\ltwo^2
\nonumber
\\[-8.5pt]
\\[-8.5pt]
\eqntext{\mbox{subject to }  \|u\|_\linf\leq\lambda, \  (DW)\T u=0.}
\end{eqnarray}
To find a solution path of~\eqref{eq:xdual3}, the KKT conditions
\eqref{eq:kkt} need to be modified to incorporate the new equality
constraint, becoming
\[
(\tD\tD\T u)_i - (\tD\ty)_i +
\alpha\gamma_i + (DW\delta)_i = 0  \qquad  \mbox{for }
i=1,\ldots, m,
\]
where the variables are $u,\alpha,\gamma,\delta$, subject
to the same constraints as before, (\ref{eq:a})--(\ref{eq:e}), and additionally
$(DW)\T u = 0$.
Instead of simply using the appropriate least squares estimate at each
iteration, we now need to solve for $u$ and $\delta$ together.
When $\lambda=\infty$, this case be done by solving the
block system
%
\begin{equation}
\label{eq:block}
\left[
\matrix{\displaystyle
\tD\tD\T& DW \cr\displaystyle
(DW)\T& 0
}
\right]
\left[
\matrix{\displaystyle
u \cr\displaystyle  \delta
}
\right]
=
\left[
\matrix{\displaystyle
\tD\ty\cr\displaystyle  0
}
\right],
\end{equation}
and in future iterations the expressions are similar.
Having done this, satisfying the rest of the
constraints (\ref{eq:a})--(\ref{eq:e})
can be done by finding the hitting and
leaving times just as we did previously.


\section{\texorpdfstring{Computational considerations.}{Computational considerations}}
\label{sec:comp}

We discuss an efficient implementation of Algorithm~\ref{alg:d},
which gives the solution path of the signal approximation problem
\eqref{eq:dlassoi}, after applying the transformation
\eqref{eq:primaldual} from dual to primal variables.
For a design with $\operatorname{rank}(X)=p$, we can modify
$y$ and $X$, and then the same algorithm gives the solution path of
\eqref{eq:dlasso}, this time relying on the transformation
\eqref{eq:xprimaldual2} for the primal path.

At each iteration of the algorithm, the dominant work is in
computing expressions of the form
\[ (D_{-\cB}(D_{-\cB})\T )^+
D_{-\cB}x
\]
for
some vector $x$, where $\cB$ is the current boundary
set [see equations~\eqref{eq:dhittime} and~\eqref{eq:dpreleave}].
Equivalently, the complexity of each iteration is based on finding
%
\begin{equation}
\label{eq:form}
\argmin_v  \Bigl\{ \|v\|_\ltwo \dvtx   v = \argmin_w
\|x-(D_{-\cB})\T w\|_\ltwo \Bigr\},
\end{equation}
the least squares solution with the smallest $\ell_2$ norm.
In the next iteration, $D_{-\cB}$ has either one less or one more row
(depending on whether a coordinate hit or left the boundary).

We can exploit the fact that the problems~\eqref{eq:form} are highly
related from one iteration to the next (our strategy that is similar
to that in the LARS implementation).
Suppose that when $\cB=\varnothing$, we solve the problem
\eqref{eq:form} by using a matrix factorization
(e.g., a QR decomposition). In future
iterations, this factorization can be efficiently updated after a row
has been deleted from or added to $D_{-\cB}$. This allows us to
compute the new solution of~\eqref{eq:form} with much less work than
it would take to solve the problem from ``scratch.''

Recall that $D$ is $m\times n$, and the dual variable $u$ is
$m$-dimensional. Let $T$ denote the number of iterations taken by the
algorithm (note that $T \geq m$, and can be strictly greater if dual
coordinates leave the boundary). When $m \leq n$, we can use
a QR factorization of $D\T$ to compute the full dual solution path in
\[
O(mn^2 + Tm^2)
\]
operations. When $m>n$, using a QR factorization of $D$ allows us to
compute the full dual solution path in
\[
O(m^2 n + Tn^2)
\]
operations. The main idea behind this implementation is fairly
straightforward. However, the details become somewhat complicated
because we require the minimum $\ell_2$ norm solution
\eqref{eq:form}, instead of a generic solution, to the least squares
problem at each iteration. See Chapters 5 and 12 of~\cite{golub} for
an extensive coverage of the QR decomposition.

We mention two simple points to improve practical efficiency:
\begin{itemize}
\item The algorithm starts at the fully regularized end of the path
($\lambda=\infty$) and works toward the unregularized solution
($\lambda=0$). Therefore, for problems in which the highly
or moderately regularized solutions are the only ones of interest,
the algorithm can compute part of the path and terminate
early. This could end up being a large savings in practice.

\item One can obtain an approximate solution path by not permitting
dual coordinates to leave the boundary (achieved by setting
$l_{k+1}=0$ in Step 3 of Algorithm~\ref{alg:d}). This makes $T=m$,
and so computing this approximate path only requires $O(mn^2)$ or
$O(m^2n)$ operations when $m \leq n$ or $m > n$, respectively. This
approximation can be quite accurate if the number times a dual
coordinate leaves the boundary is (relatively) small. Furthermore,
its legitimacy is supported by the following fact: for $D=I$, this
approximate path is exactly the LARS path when LARS is run it its
original (unmodified) state. We discuss this in the next section.
\end{itemize}

Finally, it is important to note that if one's goal is to find the
solution of~\eqref{eq:dlassoi} or~\eqref{eq:dlasso} over a discrete
set of $\lambda$ values, and the problem size is very large, then it
is likely that our path algorithm is not the most efficient
approach. The reason here is the same reason that LARS is not
generally used to solve large-scale lasso problems:
the set of critical points (changes in slope) in the piecewise
linear solution path $\hbeta_\lambda$ becomes very dense as the
problem size increases. For solving\vadjust{\goodbreak} a large problem at a fixed
$\lambda$, it is preferable to use a convex optimization technique
that was specifically developed for the purposes of
computational efficiency. First-order methods, for example, can
efficiently solve large-scale instances of~\eqref{eq:dlassoi} or
\eqref{eq:dlasso} for $\lambda$ in a discrete set (see~\cite{nesta} as an example).

Another optimization method of recent interest is coordinate
descent~\cite{cd}, which is quite efficient in solving the
lasso at discrete values of $\lambda$~\cite{pco}, and is
favored for its simplicity.
But coordinate descent cannot be used for the
minimizations~\eqref{eq:dlassoi} and~\eqref{eq:dlasso}, because the
penalty term $\|D\beta\|_\lone$ is not separable in $\beta$, and
therefore coordinate descent does not necessarily converge.
In the important signal approximation case~\eqref{eq:dlassoi},
however, the dual problem~\eqref{eq:dual} is separable, so coordinate
descent will converge if applied to the dual.
Furthermore, for various applications,
the matrix $D$ is sparse and structured, which means that the
coordinate-wise updates for~\eqref{eq:dual} are very fast.
This makes coordinate descent on the dual a promising method
for solving many of the signal approximation problems from Section~\ref{sec:apps}.

\section{\texorpdfstring{Connection to LARS.}{Connection to LARS}}
\label{sec:lars}

In this section, we return to the LARS algorithm, described in
the \hyperref[intro]{Introduction} as a point of motivation for our work. We assume
that $\operatorname{rank}(X)=p$ and $D=I$, so that~\eqref{eq:dlasso} is just
the standard lasso problem. Our algorithm gives the lasso path
$\hbeta_\lambda$, via the dual path $\hu_\lambda$; another way of
finding the lasso path is to use the LARS algorithm in its ``lasso''
mode. Since the problem is strictly convex ($X$ has full column rank),
there is only one solution at each $\lambda$, so of course these two
algorithms must give the same result.

In its original or unmodified state, LARS returns a different
path, obtained by selecting variables in order continuously decrease
the maximal absolute correlation with the residual. We refer to this
as the ``LARS path.'' Interestingly, the LARS path can be
viewed as an approximation to the lasso path (see~\cite{lars} for an
elegant interpretation and discussion of this). In our framework, we
can obtain an approximate dual solution path if we never check for
dual coordinates leaving the boundary, which can be achieved by
dropping Step 3 from Algorithm~\ref{alg:d} (or more precisely, by
setting $l_{k+1}=0$ for each $k$). If we denote the resulting dual
path by $\tu_\lambda$, then this suggests a primal path
%
\begin{equation}
\label{eq:tbeta}
\tbeta_\lambda=(X\T X)^{-1}(X\T y - \tu_\lambda),
\end{equation}
based on the transformation in~\eqref{eq:xprimaldual}. The question
is: how does this approximate solution path $\tbeta_\lambda$ compare
to the LARS path?

Figure~\ref{fig:lar} shows the two paths in question. On the left is
the familiar plot of~\cite{lars}, showing the LARS path
for the ``diabetes data.'' The colored dots on the $x$-axis mark when
variables enter the model. The right plot shows our approximate
solution path on this same data set, with vertical dashed lines
marking when variables (coordinates) hit the boundary.
The paths look identical, and this is not a coincidence:
we can show that our approximate path, which is given by
ignoring dual coordinates leaving the boundary, is equal to the LARS
path in general.

%
\begin{figure}

\includegraphics{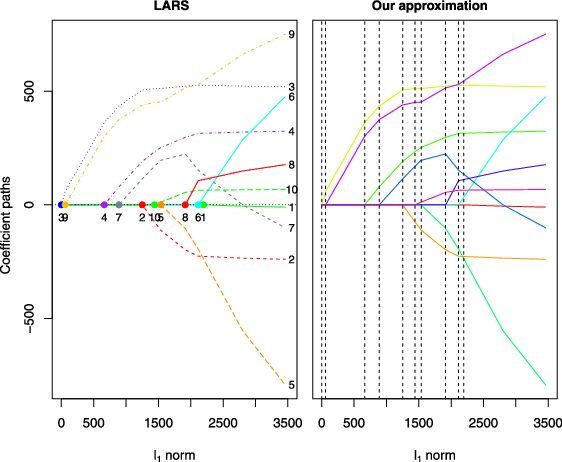}

\caption{Comparing the LARS path and our approximate lasso
path, on the diabetes data. For this data set
$n=442$ and $p=10$. The paths by parametrized by the $\ell_1$ norm
of their (respective) coefficient vectors, because the LARS path is
not naturally parametrized by $\lambda$.}
\label{fig:lar}
\end{figure}

\begin{lemma}[(Equivalence to LARS)]
\label{lemma:lars}
Suppose that $\operatorname{rank}(X)=p$ and consider using Algorithm
\ref{alg:d} to compute an approximate lasso path in the following way:
we use $\ty=XX^+y$,   $\tD=X^+$ in place of $y,D$, and we ignore Step 3
(i.e., set $l_{k+1}=0$). Let $\tu_\lambda$ denote the
corresponding dual path, and define a primal path $\tbeta_\lambda$
according to~\eqref{eq:tbeta}. Then $\tbeta_\lambda$ is exactly the
LARS path.
\end{lemma}

\begin{pf}
First, define the residual $r_\lambda=y-X\tbeta_\lambda$.
Notice that by rearranging~\eqref{eq:tbeta}, we get
$\tu_\lambda= X\T r_\lambda$.
Therefore, the coordinates of the dual path are equal to the inner
products of the columns of $X$ with the current residual. This is the
same as the correlations of the columns with the current residual,
provided we center and scale $X$ appropriately. Hence, we have a
procedure that:
\begin{itemize}
\item moves in a direction so that the absolute correlation with the
current residual is constant within $\cB$ (and maximal among all
variables) for all $\lambda$;
\item adds variables to $\cB$ once their absolute correlation with the
residual matches that realized in $\cB$.\vadjust{\goodbreak}
\end{itemize}
This almost proves that $\tbeta_\lambda$ is the LARS path, with $\cB$
being the ``active set'' in LARS terminology. What remains to be shown
is that the variables not in $\cB$ are all assigned zero coefficients.
But, recalling that $D=I$, the same arguments given in Section
\ref{sec:dprops} and Section~\ref{sec:xfull} apply here to give that
$\tbeta_\lambda\in\operatorname{null}(I_{-\cB})$ (really, $\tu_\lambda$
still solves a sequence of least squares problems, and the only
difference between $\tu_\lambda$ and $\hu_\lambda$ is in how we
construct $\cB$). This means that $\tbeta_{\lambda,-\cB}=0$, as
desired.
\end{pf}

\section{\texorpdfstring{Degrees of freedom.}{Degrees of freedom}}
\label{sec:df}

In general, the concept of degrees of freedom is of great
interest. It describes the effective number of parameters used by a
fitting procedure. This is usually easy to compute for linear
procedures (linear in the data $y$) but difficult for
nonlinear, adaptive procedures.
In this section, we derive the degrees of freedom of the fit
of problem~\eqref{eq:dlasso}, when
$\operatorname{rank}(X)=p$ and $D$ is an arbitrary penalty matrix.
This produces corollaries on degrees of freedom for various
problems presented in Section~\ref{sec:apps}. We then briefly discuss
model selection using these degrees of freedom results, and last we
discuss the role of shrinkage, a fundamental property of $\ell_1$
regularization.

\subsection{\texorpdfstring{Degrees of freedom results.}{Degrees of freedom results}}

We assume that the data $y$ is drawn from the normal model
\[
y \sim N(\mu,\sigma^2 I),
\]
and the design matrix $X$ is fixed (nonrandom). For a function $g\dvtx
\R^n \rightarrow\R^n$, with $i$th coordinate function $g_i\dvtx  \R^n
\rightarrow\R$, the degrees of freedom
of $g$ is defined as
\[
\df(g) = \frac{1}{\sigma^2} \sum_{i=1}^n
\operatorname{Cov} (g_i(y),y_i ).
\]
For our problem, the function of interest is
$g(y)=X\hbeta_\lambda(y)$, for fixed $\lambda$.


An alternative and convenient formula for
degrees of freedom comes from Stein's unbiased risk estimate
\cite{stein}. If $g$ is continuous and almost differentiable,
then Stein's formula states that
%
\begin{equation}
\label{eq:stein}
\frac{1}{\sigma^2} \sum_{i=1}^n
\operatorname{Cov} (g_i(y),y_i ) =
\E[(\nabla\cdot g)(y)].
\end{equation}
Here $\nabla \cdot g = \sum_{i=1}^n
\partial g_i/\partial y_i$
is called the divergence of $\theta$. This is useful because
typically the right-hand side of~\eqref{eq:stein} is easier to
calculate; for our problem this is the case. But using Stein's formula
requires checking that the function is continuous and almost
differentiable.
In addition to checking these regularity conditions for
$g(y)=X\hbeta_\lambda(y)$, we establish below that for almost every
$y$ the fit $X\hbeta_\lambda(y)$ is a locally affine projection.
Essentially, this allows us to take the divergence in
\eqref{eq:dlassosol2} when $X=I$, or~\eqref{eq:dlassosol4} for the
general $X$ case, and treat $\cB$ and $s$ as constants.\vadjust{\goodbreak}

As in our development of the path algorithm in Sections~\ref{sec:1d},
\ref{sec:d}  and~\ref{sec:x}, we first consider the case $X=I$,
because it is easier to understand. Notice that we can express the
dual fit as $D\T \hu_\lambda (y) = P_{C_\lambda}(y)$, the projection
of $y$ onto the convex polytope
$C_\lambda = \{D\T u\dvtx \|u\|_\linf \leq \lambda\} \subseteq \R^n$. From
\eqref{eq:primaldual}, the primal solution is just
the residual from this projection, $\hbeta_\lambda(y) =
(I-P_{C_\lambda})(y)$. The projection map onto a convex set is always
a contraction, and in fact, so is the residual from projecting\vspace*{1pt} onto a
convex set (e.g., see the proof of Theorem 1.2.2 in
\cite{schneider}). Therefore $\hbeta_\lambda(y)$ is a contraction, and
hence both continuous and almost differentiable (this follows from the
standard proof a result called ``Rademacher's theorem;'' e.g.,
see Theorem~2 in Section~3.2 of~\cite{evans}).

\begin{figure}

\includegraphics{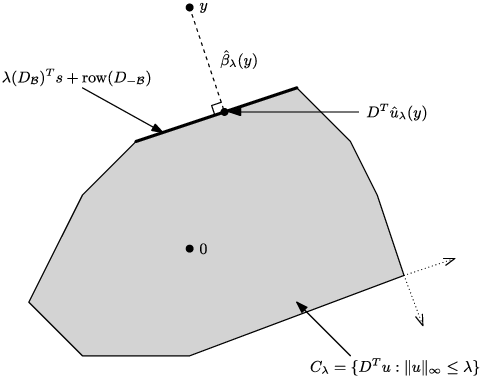}

\caption{An illustration of the geometry surrounding $\hu_\lambda$ and
 $\hbeta_\lambda$, for the case $X=I$. Recall that $\hbeta_\lambda(y)
 =y-D\T\hu_\lambda(y)$, where $D\T\hu_\lambda(y)$ is the projection
 of $y$ onto the convex\vspace*{1pt} polytope $C_\lambda=\{D\T u\dvtx  \|u\|_\linf \leq
 \lambda\}$. Almost everywhere, small pertubations of $y$ do not
 change the face on which its projection lies. The exceptional set
 $\cN_\lambda$ of points for which this property does not hold has
 dimension $n-1$, and is a union of rays like the two drawn
 as dotted lines in the bottom right of the figure.}
\label{fig:geom}
\end{figure}

Furthermore, thinking geometrically about the projection map onto
$C_\lambda$ yields a crucial insight. Examine Figure
\ref{fig:geom}---as drawn, it is clear that we can move the point $y$
slightly and it still projects to the same face of $C_\lambda$. In
fact, it seems that the only points $y$ for which this property does
not hold necessarily lie on rays that emanate orthogonally
from the corners of $C_\lambda$ (two such rays are drawn leaving the
bottom right corner). In other words, we are lead to believe
that for almost every~$y$, the projection map onto
$C_\lambda$ is a locally constant affine projection. This is indeed
true.

\begin{lemma}
\label{lemma:locconst}
For fixed $\lambda$, there exists a set $\cN_\lambda$ such that:
\begin{longlist}[(b)]
\item[(a)] $\cN_\lambda$ has Hausdorff dimension $n-1$, hence Lebesgue
 measure zero;
\item[(b)] for any $y \notin \cN_\lambda$, there exists a neighborhood
 $U$ of $y$ such that $P_{C_\lambda} \dvtx  U \rightarrow \mathbb{R}^n$ is
 simply the projection onto an affine subspace. In particular, the
 affine subspace is
 \begin{equation}
 \label{eq:affspace}
 \lambda (D_\cB)\T s + \operatorname{row}(D_{-\cB}),
 \end{equation}
 where $\cB$ and $s$ are the boundary set and signs for a
 solution $\hu_\lambda(y)$ of the dual problem~\eqref{eq:dual},
 \[
 \cB = \{i\dvtx |\hu_{\lambda,i}(y)| = \lambda\}  \quad  \mbox{and}  \quad
 s = \operatorname{sign}(\hu_{\lambda,\cB}(y)).
 \]
 The quantity~\eqref{eq:affspace} is well-defined in the sense that it
 is invariant under different choices of $\cB$ and $s$ (as the dual
 solution may not be unique).
\end{longlist}
\end{lemma}

The proof, which follows the intuition described above, is given in
\cite{supp}.

Hence we have the following result.

\begin{theorem}
\label{theorem:df}
For fixed $\lambda$, the solution $\hbeta_\lambda$ of the signal
approximation problem~\eqref{eq:dlassoi} has degrees of freedom
\[
\df(\hbeta_\lambda) =
\E\bigl[\operatorname{nullity}\bigl(D_{-\cB(y)}\bigr)\bigr],
\]
where the nullity of a matrix is the dimension of its null
space. The expectation here is taken over $\cB(y)$, the boundary
set of a dual solution $\hu_\lambda(y)$.
\end{theorem}

\textit{Note}: Above, we can choose any dual solution at $y$ to
construct the boundary set $\cB(y)$, because by Lemma
\ref{lemma:locconst}, all dual solutions give rise to the same
$\operatorname{null}(D_{-\cB(y)})$ (almost everywhere in $y$).

\begin{pf*}{Proof of Theorem~\ref{theorem:df}}
Consider $y \notin \cN_\lambda$, and let $\cB$ and $s$ be the boundary
set and signs of a dual solution $\hu_\lambda(y)$. By Lemma
\ref{lemma:locconst}, there is a neighborhood $U$ of $y$ such that
\[
\hbeta_\lambda(y') = (I - D\T \hu_\lambda)(y') =
P_{\operatorname{null}(D_{-\cB})}\bigl(y'-\lambda(D_\cB)^T s\bigr)
\]
for all $y' \in U$. Taking the divergence at $y$ we get
\[
(\nabla \cdot \hbeta_\lambda)(y) = \tr\bigl(P_{\operatorname{null}(D_{-\cB})}\bigr) =
\operatorname{nullity}(D_{-\cB}),
\]
since the trace of a projection matrix is just its rank.
This holds for almost every $y$ because $\cN_\lambda$ has measure
zero, and we can use Stein's formula to conclude that
$\df(\hbeta_\lambda) = \E[\operatorname{nullity}(D_{-\cB(y)})]$.
\end{pf*}

Now if we consider problem~\eqref{eq:dlasso}, with the design matrix\vspace*{1pt}
satisfying\break \mbox{$\operatorname{rank}(X)=p$}, then it turns out that the same
degrees of freedom formula holds for the fit $X\hbeta_\lambda$. This
is relatively straightforward to show, but requires sorting out
the details of how to turn statements involving $\ty,\tD$ into those
involving $y,D$. First, by the same arguments as
before, we know that $X\hbeta_\lambda(\ty)$ is contracting as a
function of~$\ty$. But $\ty=P_{\mathrm{col}(X)}(y)$ is contracting in
$y$, so indeed $X\hbeta_\lambda(y)$ is contracting, hence continuous
and almost differentiable, as a function of~$y$.\vadjust{\goodbreak}

Next we must establish that $\tD\T\hu_\lambda(y)$ is a locally affine
projection for almost every $y$. Well, by Lemma
\ref{lemma:locconst}, this is true of $\tD\T\hu_\lambda(\ty)$ for
$\ty \notin \cN_\lambda$, so we have the desired result except
on $\cM_\lambda=(P_{\mathrm{col}(X)})^{-1}(\cN_\lambda)$. Following
the arguments in the proof of Lemma~\ref{lemma:locconst}, it is not
hard to see that $\cN_\lambda$ now has dimension $p-1$, so $\cM_\lambda$
has measure zero.

With these properties satisfied, we have the following result.

\begin{theorem}
\label{theorem:xdf}
Suppose that $\operatorname{rank}(X)=p$.
For fixed $\lambda$, the fit $X\hbeta_\lambda$ of the generalized
lasso~\eqref{eq:dlasso} has degrees of freedom
\[
\df(X\hbeta_\lambda) =
\E\bigl[\operatorname{nullity}\bigl(D_{-\cB(y)}\bigr)\bigr],
\]
where $\cB(y)$ is the boundary set of a dual solution
$\hu_\lambda(y)$.
\end{theorem}

\textit{Note}: As before, we can construct the boundary set $\cB(y)$ from
any dual solution at $y$, because the quantity
$\operatorname{null}(D_{-\cB(y)})$ is invariant (almost everywhere in $y$).

\begin{pf*}{Proof of Theorem~\ref{theorem:xdf}}
Let $y \notin \cM_\lambda$. We need to show that
$(\nabla \cdot\break X\hbeta_\lambda)(y)=
\operatorname{nullity}(D_{-\cB(y)})$, and then applying
Stein's formula (along with the fact that $\cM_\lambda$ has measure
zero) gives the result.

Let $\cB$ denote the boundary set of a dual solution
$\hu_\lambda(y)$. Then the fit is
\[
X\hbeta_\lambda(y) =
P_{\operatorname{null}(\tD_{-\cB})}
P_{\mathrm{col}(X)} y + c,
\]
where $c$ denotes the terms that have zero derivative with
respect to $y$. Using the fact
$\operatorname{null}(X^+)=\operatorname{null}(X\T)$  and
$\operatorname{null}(\tD_{-\cB}) \supseteq \operatorname{null}(X^+)$,
\begin{eqnarray*}
P_{\operatorname{null}(\tD_{-\cB})}
P_{\mathrm{col}(X)}
&=& P_{\operatorname{null}(\tD_{-\cB})} -
P_{\operatorname{null}(\tD_{-\cB})}
P_{\operatorname{null}(X^+)} \\
&=& P_{\operatorname{null}(\tD_{-\cB})} -
P_{\operatorname{null}(X^+)}.
\end{eqnarray*}
Therefore, computing the divergence:
\begin{eqnarray*}
\big(\nabla\cdot X\hbeta_\lambda\big)(y)
&=& \operatorname{nullity}(D_{-\cB}X^+) -
\operatorname{nullity}(X^+) \\
&=& \operatorname{nullity}(D_{-\cB}),
\end{eqnarray*}
where the last equality follows because $X$ has
full column rank. This completes the proof.
\end{pf*}

We saw in Section~\ref{sec:dprops} that the null space of $D$ has a
nice interpretation for the fused lasso problem. In this case, the
theorem also becomes easier to interpret.

\begin{corollary}[(Degrees of freedom of the fused lasso)]
\label{cor:fuse}
Suppose that $\operatorname{rank}(X)=p$ and that $D$
corresponds to the fused lasso penalty on an arbitrary graph.
Then for fixed $\lambda$, the fit $X\hbeta_\lambda$ of
\eqref{eq:dlasso} has degrees of freedom
\[
\df(X\hbeta_\lambda)=
\E[\mbox{number of fused groups in } \hbeta_\lambda(y)].\vadjust{\goodbreak}
\]
\end{corollary}

\begin{pf}
If $\cG$ denotes the graph, we showed in Section~\ref{sec:dprops}
that the nullity of $D_{-\cB(\lambda,y)}$ is the number of
connected components in $\cG_{-\cB(\lambda,y)}$. We also showed (see
Section~\ref{sec:xfull} for the extension to a general design $X$)
that the coordinates of $\hbeta_\lambda(y)$ are fused on
the connected components of $\cG_{-\cB(\lambda,y)}$, giving the
result.
\end{pf}

By slightly modifying the penalty matrix, we can derive the degrees of
freedom of the sparse fused lasso.

\begin{corollary}[(Degrees of freedom of the sparse fused
lasso)]
\label{cor:sfuse}
Suppose that $\operatorname{rank}(X)=p$ and write $X_i$ for the $i$th row of
$X$. Consider the sparse fused lasso problem:
%
\begin{equation}
\label{eq:sfuse}
\minimize_{\beta\in\R^p}
\sum_{i=1}^n (y_i - X_i\T\beta)^2 + \lambda_1 \sum_{i=1}^p |\beta_i|
+ \lambda_2   \sum_{(i,j) \in E} |\beta_i-\beta_j|,
\end{equation}
where $E$ is an arbitrary set of edges between nodes
$\beta_1,\ldots,\beta_p$. Then for fixed $\lambda_1,\lambda_2$, the
fit $X\hbeta_{\lambda_1,\lambda_2}$ of~\eqref{eq:sfuse} has degrees of
freedom
\[
\df(X\hbeta_{\lambda_1,\lambda_2})=
\E[\mbox{number of nonzero fused groups in }
\hbeta_{\lambda_1,\lambda_2}(y)].
\]
\end{corollary}

\begin{pf}
We can write~\eqref{eq:sfuse} in the generalized lasso framework by
taking $\lambda=\lambda_2$ and
\[
D = \left[
\matrix{\displaystyle
D_\mathrm{fuse} \vspace*{3pt}\cr\displaystyle
\frac{\lambda_1}{\lambda_2} I
}
\right],
\]
where $D_\mathrm{fuse}$ is the fused lasso matrix corresponding to the
underlying graph, with each row giving the difference
between two nodes connected by an edge.

In Section~\ref{sec:dprops}, we analyzed the null space of
$D_\mathrm{fuse}$ to interpret the primal-dual
correspondence for the fused lasso. A similar interpretation can be
achieved with $D$ as defined above. Let $\cG$ denote the underlying
graph and suppose that it has $m$ edges (and $p$ nodes), so that
$D_\mathrm{fuse}$ is $m\times p$ and $D$ is $(m+p)\times p$. Also,
suppose that we decompose the boundary set as $\cB= \cB_1\cup\cB_2$,
where $\cB_1$ contains the dual coordinates in $\{1,\ldots, m\}$ and
$\cB_2$ contains those in $\{m+1,\ldots, m+p\}$. We can associate the
first $m$ coordinates with the $m$ edges, and the last $p$ coordinates
with the $p$ nodes. Then the matrix $D_{-\cB}$ defines a
subgraph $\cG_{-\cB}$ that can be constructed as follows:
\begin{longlist}[(2)]
\item[(1)] delete the edges of $\cG$ that correspond to
coordinates in $\cB_1$, yielding~$\cG_{-\cB_1}$;
\item[(2)] keep only the nodes of $\cG_{-\cB_1}$ that correspond to
coordinates in $\cB_2$, yielding $\cG_{-\cB}$.
\end{longlist}
It is straightforward to show that the nullity of $D_{-\cB}$ is the
number of connected components in $\cG_{-\cB}$. Furthermore,\vadjust{\goodbreak} the
solution $\hbeta_{\lambda_1,\lambda_2}(y)$ is fused on each connected
component of $\cG_{-\cB}$ and zero in all other coordinates. Applying
Theorem~\ref{theorem:xdf} gives the result.
\end{pf}

The above corollary proves a conjecture
of~\cite{fuse}, in which the authors hypothesize that the degrees of
freedom of the sparse 1d fused lasso fit is equal to the number of
nonzero fused coordinate groups, in expectation. But Corollary
\ref{cor:sfuse} covers any underlying graph, which makes it a much
more general result.

By examining the null space of $D_{-\cB}$ for other applications, and
applying Theorem~\ref{theorem:xdf}, one
can obtain more corollaries on degrees of freedom. We omit the details
for the sake of brevity, but list some such results in Table
\ref{table:df}, along with those on the fused lasso for the sake of
completeness.
The table's first result, on the degrees of freedom of the lasso, was
already established in~\cite{lassodf}. The results on
trend filtering and outlier detection can actually be derived
from this lasso result, because these problems correspond to the case
$\operatorname{rank}(D)=m$, and can be transformed into a regular
lasso problem~\eqref{eq:dlasso4}.
For the outlier detection problem, we actually need to make a
modification in order for the design matrix to have full column
rank. Recall the problem formulation~\eqref{eq:outlie2}, where the
coefficient vector is $(\alpha,\beta)\T$, the first block concerning
the outliers, and the second the regression coefficients. We set
$\alpha_1=\cdots=\alpha_p=0$, the interpretation
being that we know a priori $p$ points $y_1,\ldots, y_p$ come from
the true model, and only rest of the points $y_{p+1},\ldots, y_n$ can
possibly be outliers (this is quite reasonable for a method that
simultaneous performs a $p$-dimensional linear regression and detects
outliers).

%
\begin{table}
\tabcolsep=0pt
\tablewidth=300pt
\caption{Corollaries of Theorem \protect\ref{theorem:xdf}, giving
unbiased estimates of $\df(X\hbeta_\lambda)$ for various
problems discussed in Section \protect\ref{sec:apps}. These assume that
$\operatorname{rank}(X)=p$}
\label{table:df}
\begin{tabular*}{300pt}{@{\extracolsep{\fill}}ll@{}}
\hline
\textbf{Problem} & \multicolumn{1}{l@{}}{\textbf{Unbiased estimate of} $\bolds{\df(X\hbeta_\lambda)}$} \\
\hline
Lasso
& Number of nonzero coordinates \\
Fused lasso
& Number of fused groups \\
Sparse fused lasso
& Number of nonzero fused groups \\
Polynomial trend filtering, order $k$
& Number of knots ${}+{}k+1$ \\
Outlier detection
& Number of outliers $+{}p$ \\
\hline
\end{tabular*}
\end{table}

\subsection{\texorpdfstring{Model selection.}{Model selection}}

Note that the estimates in Table~\ref{table:df} are all easily
computable from the solution vector $\hbeta_\lambda$. The estimates for
the lasso, (sparse) fused lasso, and outlier detection problems can be
obtained by simply counting the appropriate quantity in
$\hbeta_\lambda$. The
estimate for trend filtering may be difficult to determine visually,
as it may be difficult to identify the knots in a
piecewise polynomial by eye, but the knots can counted from the
nonzeros of $D\hbeta_\lambda$. All of this is important because
it means that we can readily use model selection criteria like
$C_p$ or BIC for these problems, which employ degrees of freedom
to assess risk. For example, for the estimate $X\hbeta_\lambda$ of the
underlying mean $\mu$, the $C_p$ statistic is
\[
C_p (\lambda) = \|y-X\hbeta_\lambda\|_\ltwo^2 -n\sigma^2 +
2\sigma^2\,\df(X\hbeta_\lambda),
\]
and is an unbiased estimate of the true risk
$\E [\|\mu-X\hbeta_\lambda\|_\ltwo^2 ]$. Hence, we can define
\[
\widehat{C}_p (\lambda) = \|y-X\hbeta_\lambda\|_\ltwo^2 -n\sigma
^2 +
2\sigma^2\operatorname{nullity}(D_{-\cB}),
\]
replacing $\df(X\hbeta_\lambda)$ by its own unbiased estimate
$\operatorname{nullity}(D_{-\cB})$.
This modified statistic $\widehat{C}_p(\lambda)$ is
still unbiased as an estimate of the true risk, and this suggests
choosing $\lambda$ to minimize $\widehat{C}_p(\lambda)$. For this
task, it turns out that $\widehat{C}_p(\lambda)$\vspace*{-1pt} obtains its minimum
at one of the critical points $\{\lambda_1,\ldots,\lambda_T\}$ in the
solution path of $\hbeta_\lambda$. This is true because
$\operatorname{nullity}(D_{-\cB})$ is a step function over these critical
points, and the residual sum of squares
$\|y-X\hbeta_\lambda\|_\ltwo^2$ is monotone nondecreasing for
$\lambda$ in between critical points [this can be checked using
\eqref{eq:dlassosol4}]. Therefore, Algorithm~\ref{alg:d} can be used to
simultaneously compute the solution path and select a model, by
simply computing $\widehat{C}_p(\lambda_k)$ at each iteration $k$.

\subsection{\texorpdfstring{Shrinkage and the $\ell_1$ norm.}{Shrinkage and the l1 norm}}

At first glance, the results in Table~\ref{table:df} seem
both intuitive and unbelievable. For the fused lasso, for example, we
are told that on average we spend a single degree of freedom on each
group of coordinates in the solution. But these groups are
being adaptively selected based on the data, so aren't we using more
degrees of freedom in the end? As another example, consider the trend
filtering result: for a cubic fit, the degrees of
freedom is the number of knots $+{}4$, in expectation. A cubic
regression spline also has degrees of freedom equal to the
number of knots $+{}4$; however, in this case we fix the knot
locations ahead of time, and for cubic trend filtering the knots
are selected automatically. How can this be?

This seemingly remarkable property---that searching for the nonzero
coordinates, fused groups, knots, or outliers does not cost us anything
in terms of degrees of freedom---is explained by the shrinking
nature of the $\ell_1$ penalty. Looking back at the criterion in
\eqref{eq:dlasso}, it is not hard to see that the nonzero entries
in $D\hbeta_\lambda$ are shrunken toward zero (imagine the problem in
constrained form, instead of Lagrange form).
For the fused lasso, this means that once the groups are
``chosen,'' their coefficients are shrunken towards each other, which is less
greedy than simply fitting the group coefficients to minimize the
squared error term. Roughly speaking, this makes up for the fact that
we chose the fused groups adaptively, and in expectation, the degrees
of freedom turns out ``just right'': it is simply the number of
groups.

This leads us to think about the $\ell_0$-equivalent of problem
\eqref{eq:dlasso}, which is achieved by replacing the $\ell_1$ norm
by an $\ell_0$ norm (giving\vadjust{\goodbreak} best subset regression when
\mbox{$D=I$}). Solving this problem requires a combinatorial
optimization, and this makes it difficult to study the properties of
its solution in general. However, we do know that the solution of the
$\ell_0$ problem does not enjoy any shrinkage
property like that of the lasso solution: if we fix which entries of
$D\beta$ are nonzero, then the penalty term is constant and the
problem reduces to an equality-constrained regression. Therefore, in
light of our above discussion, it seems reasonable to conjecture that
the $\ell_0$ fit has more than $\E[\operatorname{nullity}(D_{-\cB})]$
degrees of freedom. When $D=I$, this would mean that the degrees of
freedom of the best subset regression fit is more than the number of
nonzero coefficients, in expectation.

\section{\texorpdfstring{Discussion.}{Discussion}}
\label{sec:discuss}

We have studied a generalization of the lasso problem, in which the
penalty is $\|D\beta\|_1$ for a matrix $D$. Several
important problems (such as the fused lasso and trend filtering) can
be expressed as a special case of this, corresponding to a particular
choice of $D$. We developed an algorithm to compute a solution
path for this general problem, provided that the design matrix $X$ has
full column rank. This is achieved by instead solving the (easier)
Lagrange dual problem, which, using simple duality theory, yields a
solution to the original problem after a linear transformation.

Both the dual solution path and the original solution path are
continuous\vspace*{1pt} and piecewise linear with respect to $\lambda$. The
original solution $\hbeta_\lambda$ can be written explicitly in terms
of the boundary set $\cB$, which contains the coordinates of
the dual solution that are equal to $\pm\lambda$, and the signs of
these coordinates $s$. Furthermore, viewing the dual solution
as a projection onto a convex set, we derived a simple formula for the
degrees of freedom of the generalized lasso fit.
This formula emphasizes the importance of the dual perspective, as it
is fundamentally tied to the boundary set $\cB$.
For the fused lasso problem, this result reveals that the number of
nonzero fused groups in the solution is an unbiased estimate of the
degrees of freedom of the fit, and this holds true for any underlying
graph structure. Other corollaries follow, as well.

An implementation of our path algorithm, following the ideas presented
in Section~\ref{sec:comp}, is a direction for future work, and will be
made available as an R package ``genlasso'' on the CRAN website
\cite{cran}.
There are several other directions for future research. We describe
three possibilities below.

\begin{itemize}
\item\textit{Specialized implementation for the fused lasso path
algorithm.} When $D$ is the fused lasso matrix corresponding to a
graph $\cG$, projecting onto the null space of $D_{-\cB}$ is
achieved by a simple coordinate-wise average on each connected
component of $\cG_{-\cB}$. It may therefore be possible to compute
the solution path $\hbeta_\lambda$ without having to use any linear
algebra, but by instead tracking the connectivity of $\cG$. This
could improve the computational efficiency of each iteration, and
could also lead to a parallelized approach (in which we work on each
connected component in parallel).

\item\textit{Number of steps until termination.} The number of steps
$T$ taken by our path algorithm, for a general $D$, is determined by how
many times dual coordinates leave the boundary. This is related to
an interesting problem in geometry studied by~\cite{donoho},
and investigating this connection could lead to a more definitive
statement about the algorithm's computational complexity.

\item\textit{Connection to forward stagewise regression.} When $D=I$, we
proved that our path algorithm yields the LARS path (when LARS is
run in its original, unmodified state) if we simply ignore dual
coordinates leaving the boundary. LARS can be modified to give
forward stagewise regression, which is the limit of forward stepwise
regression when the step size goes to zero (see~\cite{lars}).
A natural follow-up question is: can our
algorithm be changed to give this path too?
\end{itemize}

We believe that Lagrange duality deserves more
attention in the study of many convex optimization problems in
statistics. The dual problem can often have a complementary
(and interpretable) structure, which can offer both computational
benefits and novel mathematical or statistical insights into the
original problem.

\section*{\texorpdfstring{Acknowledgments.}{Acknowledgments}}
The authors thank Robert Tibshirani for his
many interesting suggestions and great support. Nick Henderson and
Michael Saunders provided valuable input with the computational
considerations. We also thank   Trevor Hastie for his
help with the LARS algorithm. Finally, we thank the
referees and especially the Editor for all of their help making this
paper more readable.

\begin{supplement}
\stitle{Proofs and technical details}
\slink[doi]{10.1214/11-AOS878SUPP} 
\sdatatype{.pdf}
\sfilename{aos878\_suppl.pdf}
\sdescription{A~supplementary document that contains a number of
proofs and technical details concerning ``The solution path of the
generalized lasso.''}
\end{supplement}

%

\printaddresses

\end{document}